\documentclass[11pt]{article}
\usepackage{latexsym,amsfonts,amssymb,amsmath,amsthm}
\usepackage{graphicx}

\usepackage[usenames,dvipsnames]{color}
\usepackage{ulem}

\begin{document}
\setlength{\baselineskip}{16pt}

\parindent 0.5cm
\evensidemargin 0cm \oddsidemargin 0cm \topmargin 0cm \textheight 22.5cm \textwidth 16cm \footskip 2cm \headsep
0cm

\newtheorem{theorem}{Theorem}[section]
\newtheorem{lemma}{Lemma}[section]
\newtheorem{proposition}{Proposition}[section]
\newtheorem{definition}{Definition}[section]
\newtheorem{example}{Example}[section]
\newtheorem{corollary}{Corollary}[section]

\newtheorem{remark}{Remark}[section]

\numberwithin{equation}{section}

\def\p{\partial}
\def\I{\textit}
\def\R{\mathbb R}
\def\C{\mathbb C}
\def\u{\underline}
\def\l{\lambda}
\def\a{\alpha}
\def\O{\Omega}
\def\e{\epsilon}
\def\ls{\lambda^*}
\def\D{\displaystyle}
\def\wyx{ \frac{w(y,t)}{w(x,t)}}
\def\imp{\Rightarrow}
\def\tE{\tilde E}
\def\tX{\tilde X}
\def\tH{\tilde H}
\def\tu{\tilde u}
\def\d{\mathcal D}
\def\aa{\mathcal A}
\def\DH{\mathcal D(\tH)}
\def\bE{\bar E}
\def\bH{\bar H}
\def\M{\mathcal M}
\renewcommand{\labelenumi}{(\arabic{enumi})}

\def\disp{\displaystyle}
\def\undertex#1{$\underline{\hbox{#1}}$}
\def\card{\mathop{\hbox{card}}}
\def\sgn{\mathop{\hbox{sgn}}}
\def\exp{\mathop{\hbox{exp}}}
\def\OFP{(\Omega,{\cal F},\PP)}
\newcommand\JM{Mierczy\'nski}
\newcommand\RR{\ensuremath{\mathbb{R}}}
\newcommand\CC{\ensuremath{\mathbb{C}}}
\newcommand\QQ{\ensuremath{\mathbb{Q}}}
\newcommand\ZZ{\ensuremath{\mathbb{Z}}}
\newcommand\NN{\ensuremath{\mathbb{N}}}
\newcommand\PP{\ensuremath{\mathbb{P}}}
\newcommand\abs[1]{\ensuremath{\lvert#1\rvert}}

\newcommand\normf[1]{\ensuremath{\lVert#1\rVert_{f}}}
\newcommand\normfRb[1]{\ensuremath{\lVert#1\rVert_{f,R_b}}}
\newcommand\normfRbone[1]{\ensuremath{\lVert#1\rVert_{f, R_{b_1}}}}
\newcommand\normfRbtwo[1]{\ensuremath{\lVert#1\rVert_{f,R_{b_2}}}}
\newcommand\normtwo[1]{\ensuremath{\lVert#1\rVert_{2}}}
\newcommand\norminfty[1]{\ensuremath{\lVert#1\rVert_{\infty}}}
\newcommand{\ds}{\displaystyle}

\title{Finite-time blow-up prevention by logistic source in parabolic-elliptic chemotaxis
models with singular sensitivity in any dimensional setting}
\author{
Halil Ibrahim Kurt and Wenxian Shen   \thanks{Partially supported by the NSF grant DMS--1645673.}\\
Department of Mathematics and Statistics\\
Auburn University\\
Auburn University, AL 36849\\
U.S.A. }

\date{}
\maketitle

\begin{abstract}
{In recent years, {a lot of attention has been drawn} to the question of whether  logistic kinetics is sufficient to enforce the global existence of classical solutions or to prevent finite-time blow-up in various chemotaxis models.
However, for several important chemotaxis models,  only  in the space two dimensional setting, it has been
shown that  logistic kinetics is sufficient to enforce the global existence of classical solutions
(see  \cite{FuWiYo1} and \cite{TeWi1}).
The current paper is to study the above question for    the following parabolic-elliptic chemotaxis system with singular sensitivity and logistic source in any space dimensional setting,
\begin{equation}
    \begin{cases}
u_t=\Delta u-\chi\nabla\cdot (\frac{u}{v} \nabla v)+u(a(x,t)-b(x,t) u^{1+\sigma}),\quad &x\in \Omega\cr
0=\Delta v-\mu v+\nu u,\quad &x\in \Omega \quad \cr
\frac{\p u}{\p n}=\frac{\p v}{\p n}=0,\quad &x\in\p\Omega,
\end{cases}
\end{equation}
where $\Omega \subset \mathbb{R}^n$ is a bounded domain with  smooth boundary $\p\Omega$, $\chi$ is the singular chemotaxis sensitivity coefficient,  $a(x,t)$ and $b(x,t)$ are positive smooth functions,
$\mu,\nu$ are positive constants, and $\sigma\ge 0$. {When} $\sigma>0$, we prove that,  for every given nonnegative initial data $0\not\equiv  u_0\in C^0(\bar \Omega)$, (0.1) has a unique globally defined   classical solution $(u_\sigma(x,t;u_0),v_\sigma(x,t;u_0))$
with $u_\sigma(x,0;u_0)=u_0(x)$,  which shows that, in any space dimensional setting,  strong logistic
kinetics is sufficient to enforce  the global existence of classical solutions and hence prevents the occurrence of finite-time blow-up even for arbitrarily
large $\chi$. In addition, the solutions are shown to be uniformly bounded under the conditions \begin{equation*}
    a_{\inf}>
\begin{cases}
\frac{\mu \chi^2}{4}, &\text{if $0< \chi \leq 2,$}\\
\mu(\chi-1), &\text{if $\chi>2$.}\\
\end{cases}
\end{equation*}
When $\sigma=0$,  we show that  the classical solution $(u(x,t;u_0,0),v(x,t;u_0,0))$ exists globally and stays bounded provided that both $a(x,t)$ and  $u_0(x)$ are not small.}
\end{abstract}

\medskip


\medskip

\medskip


\newpage
\section{Introduction and Main Results}
\label{S:intro}

Chemotaxis refers to the movement of living organisms in response to certain chemicals in their
environments, and plays a crucial role in a wide range of biological phenomena such as immune system response, embryo development, tumor growth, population dynamics, gravitational collapse, etc. (see  \cite{ISM04, DAL1991}).
Chemotaxis models, also known as Keller-Segel models, have been widely studied since the pioneering works \cite{KS1970, KS71} by  Keller and Segel at the beginning of 1970s on the mathematical modeling of  the aggregation process of Dictyostelium discoideum. One of the central problems studied in the literature on chemotaxis models is whether solutions blow up in finite time or exist globally.

In recent years, a large amount of research has been carried out toward the finite-time blow-up prevention by logistic source in various chemotaxis models. For example, consider the following chemotaxis model with logistic source,
  \begin{equation}
\label{main-eq0}
\begin{cases}
u_t=\Delta u- \chi\nabla \cdot (u\nabla v)+u(a(x,t)-b(x,t)u),\quad &x\in \Omega \cr
\tau v_t=\Delta v-\mu v+\nu u,\quad &x\in \Omega\cr
\frac{\p u}{\p n}=\frac{\p v}{\p n}=0,\quad &x\in \partial\Omega,
\end{cases}
\end{equation}
where $\Omega\subset\mathbb{R}^N$ is a bounded domain with smooth
boundary, $u(x,t)$ represents the population density of a species and $v(x,t)$ represents the population density of some chemical substance, $\chi$ is the chemotaxis  sensitivity coefficient, $a,b$ are positive continuous functions,    $\tau\ge 0$ is a non-negative constant linked to the diffusion rate of the chemical substance, and $\mu>0$  represents the degradation rate of the  chemical substance  and $\nu>0$ is the rate at which the  species produces the chemical substance.  When $\tau=0,$  $a(x,t)\equiv a$, $b(x,t)\equiv b$, $\mu=\nu=1$,  it is proved in \cite{TeWi1} that,  if $N\le 2$ or  $b>\frac{N-2}{N}\chi$, then for every nonnegative initial data $u_0\in C^0(\bar\Omega)$, \eqref{main-eq0}  possesses a global bounded classical solution which is unique. It should be pointed out that, when $a=b=0$ and $N\ge 2$,
finite-time blow-up of positive  solutions occurs under some condition on the mass and the moment of the initial data   (see \cite{HeMeVe}, \cite{HeVe},  \cite{Nag2}, \cite{NaSe3}). Hence the finite time blow-up phenomena in \eqref{main-eq0} is suppressed to some extent by  the logistic source. But it remains open whether in any space dimensional setting, for every nonnegative initial data  $u_0\in C^0(\bar\Omega)$ \eqref{main-eq0}   possesses a unique global  classical solution for every $\chi>0$ and every $b$ with $\inf_{x\in\bar\Omega,t\in\R} b>0$. It should be  pointed out that finite-time blow-up  occurs in  various variants of \eqref{main-eq0}, for example, it occurs in \eqref{main-eq0} with $\tau=0$, with  the logistic source being replaced by logistic-type superlinear degradation (see \cite{TaYo, Win5}), and/or with the second equation being replaced by the following one,
\begin{equation*}
    0=\Delta v-\frac{1}{|\Omega|}\int_\Omega u(\cdot,t)+u,\quad x\in\Omega
\end{equation*}
(see \cite{BlFuLa, Fue1, Fue2, Win2}). The reader  is referred to \cite{ChTe, IsSh, Lan0, Lan, TaWi, Tel, Vig, Win-1} and references therein for other studies on the global existence of nonnegative solutions of \eqref{main-eq0}.

Consider the following  chemotaxis system with singular sensitivity and logistic source,
\begin{equation}
\label{main-eq00}
\begin{cases}
u_t=\Delta u-\chi\nabla\cdot (\frac{u}{v} \nabla v)+u(a(x,t)-b(x,t){u^{1+\sigma}}) ,\quad &x\in \Omega\cr
\tau v_t=\Delta v-\mu v+\nu u,\quad &x\in \Omega \quad \cr
\frac{\p u}{\p n}=\frac{\p v}{\p n}=0,\quad &x\in\p\Omega.
\end{cases}
\end{equation}
When $a(x,t)\equiv a>0$,  $b(x,t)\equiv b>0$, $\tau =0$,  {$\sigma=0,$} and $\mu=\nu=1$,
it is proved in \cite{FuWiYo1} that, if $N=2$, then  \eqref{main-eq00} has a unique global solution with any nonnegative initial data $0\not\equiv u_0\in C^0(\bar \Omega)$. When $a(x,t)\equiv 0$, $b(x,t)\equiv 0$, $\tau=0$, {$\sigma=0,$} and $\mu=\nu=1$, it is proved in \cite{FuSe1} that, if $N=2$, then   \eqref{main-eq00} has a unique global solution with any nonnegative initial data $0\not\equiv u_0\in C^0(\bar \Omega)$. It should be pointed out that, when $a(x,t)\equiv 0$, $b(x,t)\equiv 0$, and $N\ge 3$, there exists some nonnegative initial data $u_0\in C^0(\bar \Omega)$  such that the solution of \eqref{main-eq00} blows up at some finite time (see \cite{NaSe3}). It remains open whether in any space dimensional setting, for every nonnegative initial data $0\not\equiv u_0\in C^0(\bar\Omega)$ \eqref{main-eq00}   possesses a unique global  classical solution  for every $\chi>0$ and every $b$ with $\inf_{x\in\bar\Omega,t\in\R} b>0$. The reader is referred to  \cite{Bil, Bla, DiWaZh, Fuj, FuSe2, FuSe3,  FuWiYo, FuYo, NaSe3, Win4, ZhZh, ZhZh1} and references therein for other studies on the global existence of nonnegative solutions of \eqref{main-eq00}.

The objective of this paper is to investigate the finite-time blow-up prevention by logistic source in the chemotaxis model \eqref{main-eq00} with   $\tau=0$ in any space dimensional setting, that is,
\begin{equation}
\label{main-eq}
\begin{cases}
u_t=\Delta u-\chi\nabla\cdot (\frac{u}{v} \nabla v)+u(a(x,t)-b(x,t){u^{1+\sigma}}) ,\quad &x\in \Omega\cr
0=\Delta v-\mu v+\nu u,\quad &x\in \Omega \quad \cr
\frac{\p u}{\p n}=\frac{\p v}{\p n}=0,\quad &x\in\p\Omega.
\end{cases}
\end{equation}
When $\sigma>0$, $u(a(x,t)-b(x,t)u^{1+\sigma})$ is referred to as a {\it strong logistic source}, and when $\sigma=0$, $u(a(x,t)-b(x,t)u)$ is referred to as a {\it regular logistic source}. We will  show that, for every $N\ge 1$, a strong logistic source  is sufficient to enforce global existence of positive classical solutions of \eqref{main-eq} and hence prevents the occurrence of finite-time blow-up even for arbitrarily large $\chi$. In the regular logistic source case, we will show that  finite-time blow-up does not occur  provided that $a(x,t)$ is  large  relative to $\chi$ and the initial function $u_0(x)$ is not too small.

To be more precise, we assume throughout this paper  that  $\chi, \nu$, and $\mu$ are positive constants; {$\sigma \ge 0$;} $\Omega\subset\RR^N$ is a smooth bounded domain;
  the initial function $u_0(x)$  satisfies
\begin{equation}
\label{initial-cond-eq}
u_0 \in C^0(\bar{\Omega}), \quad u_0 \ge 0, \quad {\rm and} \quad \int_\Omega u_0 >0;
\end{equation}
$a(x,t)$ and $b(x,t)$ are H\"older continuous in $t\in\RR $ with exponent $\gamma>0$ uniformly with respect to $x\in\bar\Omega$, continuous in $x\in\bar\Omega$ uniformly
with respect to $t\in\RR$, and there are positive constants $\alpha_i$, $A_i$ $(i=1,2)$ such that
{\begin{equation}
    \label{eq-1-6}
    0< \alpha_1\leq a(x,t) \leq A_1\quad \text{and} \quad
0< \alpha_2\leq {b(x,t)} \leq A_2.
\end{equation}}
We put
\begin{equation*}
a_{\inf}=\inf _{ x \in\bar{\Omega},t\in\R}a(x,t),\quad a_{\sup}=\sup _{x \in\bar{\Omega},t\in\R}a(x,t),
\end{equation*}
and
\begin{equation*}
b_{\inf}=\inf _{x \in\bar{\Omega},t\in\R}b(x,t),\quad b_{\sup}=\sup _{x \in\bar{\Omega},t\in\R}b(x,t),
\end{equation*}
unless specified otherwise. By the arguments in \cite[Lemma 2.2]{FuWiYo1}, we have the following proposition on the existence and uniqueness of the solution of \eqref{main-eq} with given initial function $u_0$ satisfying \eqref{initial-cond-eq}.

\begin{proposition}
\label{main-prop}(Local existence)
Suppose that \eqref{initial-cond-eq} holds. Then there is $T_{\max}^\sigma(u_0)\in (0,\infty]$
such that the system \eqref{main-eq} has a unique classical solution { $(u_{\sigma}(x,t;u_0),v_{\sigma}(x,t;u_0))$, on $(0,T_{\max}^\sigma(u_0))$} satisfying that
\begin{equation*}
    \lim_{t\to 0}\|{ u_{\sigma}(\cdot,t;u_0)}-u_0(\cdot)\|_{C^0(\bar\Omega)}=0,
\end{equation*}
and
\begin{equation*}
u(\cdot,\cdot;u_0)\in  C(\bar \Omega \times(0,T_{\max}(u_0)))\cap C^{2,1}( \bar \Omega \times (0,T_{\max}(u_0))),
\end{equation*}
and
\begin{equation*}
    v(\cdot,\cdot;u_0)\in  C^{2,0}( \bar \Omega \times (0,T_{\max}(u_0))).
\end{equation*}
Moreover if $T_{\max}^\sigma(u_0)< \infty,$ then
\begin{equation*}
\limsup_{t \nearrow { T_{\max}^\sigma(u_0)}} \left\| {u_{\sigma}(\cdot,t;u_0)} \right\|_{C^0(\bar \Omega)}  =\infty \quad or \quad \liminf_{t \nearrow { T_{\max}^\sigma(u_0)}} \inf_{x \in \Omega} { v_{\sigma}(\cdot,t;u_0)}=0.
\end{equation*}
\end{proposition}

For given $u_0\in C^0(\bar\Omega)$ satisfying \eqref{initial-cond-eq},  the unique solution { $(u_{\sigma}(x,t;u_0), v_{\sigma}(x,t;v_0))$} of \eqref{main-eq} with given initial function $u_0$ is said to be {\it locally mass persistent} if {for each $ 0<T< \infty$},
\begin{equation}
\label{ln-u-estimate-eq1-1}
\inf_{t\in [0,\min\{T,T_{\max}^{\sigma}(u_0)\})}\int_\Omega u_{\sigma}(x,t;u_0)dx>0,
\end{equation}
and
\begin{equation}
\label{ln-u-estimate-eq1-2}
\inf_{x\in\Omega,t\in [0,\min\{T,T_{\max}^{\sigma}(u_0)\})} v_{\sigma}(x,t;u_0)>0,
\end{equation} and {\it mass persistent} if \eqref{ln-u-estimate-eq1-1} and \eqref{ln-u-estimate-eq1-2} hold for each $0<T\le \infty$.  $(u_{\sigma}(x,t;u_0), v_{\sigma}(x,t;v_0))$ is said to be {\it globally defined} if  $T^{\sigma}_{\max}(u_0)=\infty$. If $(u_{\sigma}(x,t;u_0), v_{\sigma}(x,t;v_0))$ is globally defined, it is said to be {\it bounded} if
$\sup_{t\ge 0,x\in\Omega} u_{\sigma}(x,t;u_0)<\infty$, which implies $\sup_{t\ge 0,x\in\Omega} v_{\sigma}(x,t;u_0)<\infty$.

By properly modified arguments of \cite{FuWiYo1}, we  have the following proposition on the mass persistence
of  the unique solution of \eqref{main-eq} with given initial function $u_0$ satisfying
\eqref{initial-cond-eq}.

\begin{proposition}
\label{mass-persistence-prop} (Mass persistence)
{ Suppose that $u_0$ satisfying \eqref{initial-cond-eq}.

\begin{itemize}
\item[(1)]
For any $0<T<\infty$, \eqref{ln-u-estimate-eq1-1} and \eqref{ln-u-estimate-eq1-2} hold.

\item[(2)] If \begin{equation}
    \label{def-ainf-mu-chi}
    a_{\inf}>
\begin{cases}
\frac{\mu \chi^2}{4}, &\text{if $0< \chi \leq 2,$}\\
\mu(\chi-1), &\text{if $\chi>2$}\\
\end{cases}
\end{equation}
holds, then
\eqref{ln-u-estimate-eq1-1} and \eqref{ln-u-estimate-eq1-2} hold for any $0<T\le \infty$.
\end{itemize}}
\end{proposition}

{ The main results of this paper include the $L^p(\Omega)$-boundedness for some $p>n$ and  global existence and $L^\infty(\Omega)$-boundedness of the unique solution of \eqref{main-eq} with given initial function $u_0$ satisfying
\eqref{initial-cond-eq}. The proof of $L^p(\Omega)$-boundedness heavily replies on the following proposition,

\begin{proposition}
\label{main-prop2}
Let $p \ge  3$ and $p-1 \leq k <  2p-2.$ There exist positive constants $M^*=M^*(p,k,\nu)>0$ and $M^{**}=M^{**}(p,k,\mu,\nu)^*>0$ such that for any given $u_0$ satisfying \eqref{initial-cond-eq},
\begin{equation}
\label{main-inequality}
   \int_{\Omega} \frac{|\nabla v_\sigma(x,t;u_0)|^{2p}}{v_\sigma^{k+1}(x,t;u_0)} \leq M^* \int_{\Omega} \frac{u_\sigma^p(x,t;u_0)}{v_\sigma^{k-p+1}(x,t;u_0)}+M^{**}\int_{\Omega} v_\sigma^{2p-k-1}(x,t;u_0)
\end{equation}
for all  $t \in (0, T_{\rm max}^\sigma(u_0))$.
\end{proposition}

Let
\begin{equation}
\label{C-star-eq-00}
C^*=C^*(p,\nu):=M^*(p+1,p,\nu),\quad C^{**}=C^{**}(p,\mu,\nu):=M^{**}(p+1,p,\mu,\nu)
\end{equation}
and  $\delta^*=\delta^*(\mu,n)$ be such that for any $u_0$ satisfying (1.4) and $\sigma\ge 0$,
\begin{equation*}
v_\sigma(x,t;u_0)\ge \delta^*\nu\int_\Omega u_\sigma(x,t;u_0)dx\quad \forall\, x\in\Omega,\,\, t\in [0,T_{\max}^\sigma(u_0))
\end{equation*}
(the existence of such $\delta^*$ follows from
 Lemma \ref{prelim-lm1}).}  When $\sigma=0$, we may put
\begin{equation*}
     (u(x,t;u_0),v(x,t;u_0))=(u_0(x,t;u_0),v_0(x,t;u_0))\;\; {\rm and}\;\;T_{\max}(u_0)=T_{\max}^0(u_0).
\end{equation*}

The first main theorem is on the $L^p(\Omega)$-boundedness of  the unique solution of \eqref{main-eq} with given initial function $u_0$ satisfying
\eqref{initial-cond-eq}.

\begin{theorem}
\label{main-thm1} ($L^p$-boundedness) {Suppose that $u_0$ satisfies
\eqref{initial-cond-eq} {and $a(\cdot,\cdot), b(\cdot,\cdot)$ satisfy \eqref{eq-1-6}.}
\begin{itemize}
\item[(1)] For any given $\sigma>0$, $p> n$, and  $0<T<\infty$, there holds
\begin{equation}
\label{thm1-eq1}
\sup_{0\le t<\min\{T,T_{\max}^{\sigma}(u_0)\}}\|u_{\sigma}(\cdot,t;u_0)\|_{L^p}<\infty.
\end{equation}

\item[(2)] If \eqref{def-ainf-mu-chi} is valid, then for any $\sigma>0$ and  $p>n$,  \eqref{thm1-eq1}  also holds for $T=\infty$.

\item[(3)]  Consider \eqref{main-eq} with $\sigma=0$. Assume that
\begin{equation}
\label{a-assumption-eq1}
a_{\inf}> a_{\chi,\mu}+ \frac{b_{\sup}|\Omega| (p_n-1) \big(C_n^*\big)^{\frac{1}{p_n+1}}}{4b_{\inf} \delta^* \nu}\frac{\chi^2}{C_{\chi}},
\end{equation}
where $p_n=\max\{2,n\}$,  $C_n^*=C^*(p_n,\nu)$,
\begin{equation*}
   a_{\chi,\mu}:=2\big(\chi+2-2\sqrt{\chi+1}\big)\mu,
\end{equation*}
 and
\begin{equation*}
    C_\chi=\begin{cases} \chi,\quad {\rm if}\,\,\,\, 0<\chi\le 1\cr
1,\quad {\rm if}\,\,\,\, \chi>1.
\end{cases}
\end{equation*}
For any $u_0$ satisfying (1.4), if there is
$\tau\in [0,T_{\max}(u_0))$ such that
\begin{equation}
\label{u-0-assumption-eq}
\int_\Omega u^{-1}(x,\tau;u_0)\le \frac{b_{\sup}|\Omega|}{\big(a_{\inf}-a_{\chi,\mu}\big)C_{\chi}},
\end{equation}
then  there is $p>\max\{2,n\}$ such that
\begin{equation*}
    \sup_{0\le t<T_{\max}(u_0)}\|u_0(\cdot,t;u_0)\|_{L^p}<\infty.
\end{equation*}
\end{itemize}}
\end{theorem}

The second main theorem  is on the global existence and boundedness of the unique solution of \eqref{main-eq} with given initial function $u_0$ satisfying
\eqref{initial-cond-eq}.

\begin{theorem}
\label{main-thm2} (Global existence and boundedness)
{ Suppose that $u_0$ satisfies
\eqref{initial-cond-eq}  {and $a(\cdot,\cdot), b(\cdot,\cdot)$ satisfy \eqref{eq-1-6}.}
\begin{itemize}

\item[(1)] (Global existence) For any $\sigma>0$,  $(u_{\sigma}(x,t;u_0),v_{\sigma}(x,t;u_0))$ exists globally, that is, $$T_{\max}^{\sigma}(u_0)=\infty.$$

\item[(2)] (Boundedness) Suppose that \eqref{def-ainf-mu-chi} holds. Then for any $\sigma>0$,  $(u_{\sigma}(x,t;u_0)$, $v_{\sigma}(x,t;u_0))$  is bounded, i.e. there exists $C>0$ such that
\begin{equation*}
     \|u_{\sigma}(\cdot,t;u_0)\|_{L^{\infty}(\Omega)} \leq C \quad \quad \text{for all}\,\,  t>0.
\end{equation*}

\item[(3)] (Global existence and boundedness)  Consider \eqref{main-eq} with $\sigma=0$. Suppose that \eqref{a-assumption-eq1} and \eqref{u-0-assumption-eq} hold. Then $T_{\max}(u_0)=\infty$ and there exists $C>0$ such that
\begin{equation*}
    \|u(\cdot,t;u_0)\|_{L^{\infty}(\Omega)} \leq C \quad \quad \text{for all}\,\,  t>0.
\end{equation*}
\end{itemize}}
\end{theorem}

\begin{remark}
\begin{itemize}
\item[(1)] Theorem \ref{main-thm2}(1) implies that, in any space dimensional setting, { strong logistic source} is sufficient to enforce the global existence of positive classical solutions of the chemotaxis system \eqref{main-eq} with singular sensitivity  and hence prevents the occurrence of finite-time blow-up even for arbitrarily large $\chi$.
Its proof heavily relies on Theorem \ref{main-thm1}(1).

\item[(2)] Theorem \ref{main-thm1}(2) implies that, when $a(x,t)$ is large relative to the chemotaxis sensitivity coefficient $\chi$, the globally defined solution $(u_\sigma(x,t;u_0),v_\sigma(x,t;u_0))$ of \eqref{main-eq} with $\sigma>0$ is bounded. Its proof is based on Theorem \ref{main-thm1}(2).

{

\item[(3)] The results stated in Theorems \ref{main-thm1}(1) and Theorem \ref{main-thm2}(1)   are obtained for the first time for \eqref{main-eq}.  The results stated in Theorems \ref{main-thm2}(2) and Theorem \ref{main-thm2}(2) extend \cite[Theorem 1]{Zha}.
In \cite[Theorem 1]{Zha}, it is proved that $(u_\sigma(x,t;u_0),v_\sigma(x,t;u_0))$ exists globally and stays bounded
when $a(x,t)\equiv a$ and $b(x,t)\equiv b$ are constants functions,   \eqref{def-ainf-mu-chi}  holds, and
$\sigma>1-\frac{2}{n}$.

\item[(4)]
Observe that $a_{\chi,\mu}$ in \eqref{a-assumption-eq1} satisfies
\begin{equation*}
    a_{\chi,\mu}\le \begin{cases} \frac{\mu \chi^2}{2},\quad &{\rm if}\,\,\,\,  0<\chi\le 2\cr
2\mu(\chi-1), \quad &{\rm if}\,\,\, \,  \chi>2.
\end{cases}
\end{equation*}
The condition  \eqref{a-assumption-eq1} is independent of $b(x,t)$ when $b(x,t)\equiv b$ is a constant function. It indicates that $a(x,t)$ is large relative to $\chi$, which makes $v$ does not become too small and is a natural condition. When $\chi\to 0$,  the condition  \eqref{a-assumption-eq1} becomes $a_{\inf}>0$. An explicit formula  for $C^*(p,\nu)$ in \eqref{a-assumption-eq1}  can be given by choosing $\epsilon_1=\epsilon_2=\epsilon$ an $\epsilon=\frac{4p-3}{4(2p-1)^2}$ in the proof of Proposition \ref{main-prop2}.

\item[(5)]
The condition
\eqref{u-0-assumption-eq} indicates that $u_0$ is not too small, which makes $v$ does not become too small and is also a natural condition.   When $\chi$ is small or $b(x,t)$ is big, \eqref{u-0-assumption-eq}  is a  week condition. In fact, when $\chi\to 0$ or $b_{\sup}\to\infty$, the condition
\eqref{u-0-assumption-eq} becomes $\int_\Omega u^{-1}(x,\tau;u_0)<\infty$. Theorem \ref{main-thm1}(3) and Theorem \ref{main-thm2}(3) show the following interesting phenomenon:  in the case that $\sigma=0$, finite-time blow-up does not occur  provided that
$a(x,t)$ is  large  relative  to $\chi$ and the initial function $u_0(x)$ is not too small.}

\item[(6)]  {The inequality \eqref{main-inequality} is a fundamental inequality and is established in this paper for the first time. The proof of \eqref{main-inequality} is very  nontrivial.}
\end{itemize}
\end{remark}

The rest of this paper is organized as follows. In section 2, we present some preliminary  lemmas that will be used to prove our main results. In section 3, we study the mass persistence and prove Proposition \ref{mass-persistence-prop}.  In section 4, we explore the $L^p$ boundedness of positive solutions of \eqref{main-eq}, derive the estimate \eqref{main-inequality} ,  and prove Theorem \ref{main-thm1}. We study the global existence and boundedness of positive solutions of \eqref{main-eq},  and prove the Theorem \ref{main-thm2} in  section 5.

\section{Preliminaries}

In this section, we present some lemmas to be used in later sections.

Throughout the rest of this section, we assume that $\Omega\subset\mathbb{R}^n$ is a bounded smooth domain and
{ $(u_{\sigma}(x,t),v_{\sigma}(x,t)):=(u_{\sigma}(x,t;u_0),v_{\sigma}(x,t;u_0))$} is the unique classical solution of \eqref{main-eq} on the maximal interval $(0,T^{\sigma}_{\max}):=(0,T_{\max}^\sigma(u_0))$ with the initial function $u_0$ satisfying \eqref{initial-cond-eq}. Note that
\begin{equation*}
    { u_{\sigma}(x,t), v_{\sigma}(x,t)>0}\quad \forall\, x\in\Omega,\,\, t\in (0,T_{\max}).
\end{equation*}

Assume $1 < p < \infty.$ Let $A:\mathcal{D}(A)\subset L^p(\Omega)\to L^p(\Omega)$ be defined by
$A=-\Delta +I$  with $\mathcal{D}(A)=\{u\in W^{2,p}(\Omega)\,|\, \frac{\p u}{\p n}=0$ on $\p \Omega\}$. Then $-A$ generates an analytic semigroup on $L^p(\Omega)$. We denote it by
$e^{-tA}$. Note that ${\rm Re}\sigma(A)>0$.  Let $A^\beta$ be the fractional power operator of $A$ (see \cite[Definition 1.4.1]{Hen}).
Let $X^\beta=\mathcal{D}(A^\beta)$ with graph norm $\|u\|_{X^\beta}=\|A^\beta u\|_{L^p(\Omega)}$ for $\beta\ge 0$ and $u\in X^\beta$
(see \cite[Definition 1.4.7]{Hen}).

{ We first present some well-known estimates associated with $A$ and $e^{-At}$.}

\begin{lemma}
\label{prelim-lm1}
 Let $w \in C^0(\bar \Omega)$ be nonnegative function such that $\int_{\Omega} w >0$. If $z$ is a weak solution to
\begin{equation*}
    \begin{cases}
    -\Delta z +\mu z=w, \quad &x\in \Omega \cr
    \frac{\p z}{\p n}=0, \quad \quad \quad &x\in \p \Omega,
    \end{cases}
\end{equation*}
then
\begin{equation*}
    z \ge \delta_0 \int_{\Omega} w>0 \quad {\rm in} \quad \Omega,
\end{equation*}
where
\begin{equation*}
\delta_0=C \int_0^\infty \frac{1}{(4\pi t)^{\frac{n}{2}}}e^{-\big(\omega t+\frac{C}{4t}\big)}dt < \infty,
\end{equation*}
for some positive constants $C$ and $\omega$ depending only on $\Omega$.
\end{lemma}

\begin{proof}
It follows from the arguments of  \cite[Lemma 2.1]{FuWiYo} and the
gaussian lower bound for the heat kernel of the laplacian with Neumann boundary
condition on smooth domain (see \cite[Theorem 4]{ChOuYa}).
\end{proof}

\begin{lemma}
\label{prelim-lm2}
\begin{itemize}
\item[(i)] For each $\beta\ge 0$, there is $C_\beta>0$ such that for some $\gamma>0,$
\begin{equation*}
    \|A^\beta e^{-At}\| \leq C_\beta t^{-\beta} e^{-\gamma t} \;\; for \;\; t>0.
\end{equation*}

\item[(ii)]
If $m \in \{0,1\}$ and $q\in [p,\infty]$ are such that $m-\frac{n}{q}<2\beta-\frac{n}{p}$,
then
$$
 X^\beta\hookrightarrow W^{m,q}(\Omega).
$$

\item[(iii)] If $2\beta -\frac{n}{p}>\theta\ge 0$, then
$$
X^\beta\hookrightarrow C^\theta(\Omega).
$$
\end{itemize}
\end{lemma}

\begin{proof}
(i) It follows from \cite[Theorem 1.4.3]{Hen}.

(ii) It follows from \cite[Theorem 1.6.1]{Hen}.

(iii) It also follows from \cite[Theorem 1.6.1]{Hen}.
\end{proof}

\begin{lemma}
\label{prelim-lm3}
 Given $1<p<\infty$, there is $K=K(p)>0$ such that
\begin{equation}
\label{aux-new-eq1}
\| e^{-tA}\nabla \cdot \phi\|_{ L^p(\Omega)}\leq  K(p)(1+t^{-\frac{1}{2}}) e^{-\gamma t} \|\phi\|_{ L^p(\Omega)}
 \end{equation}
 for some $\gamma>0$,  all $t>0$,  and $\phi \in C^1(\bar{\Omega})$ satisfying $\frac{\p \phi}{\p n}=0$  on $\partial \Omega$.
  Consequently, for all $t>0$,  the operator $e^{-tA}\nabla\cdot$ possesses  a uniquely determined
extension to an operator from $L^p(\Omega)$ into $L^p(\Omega)$, with norm controlled according to \eqref{aux-new-eq1}.
\end{lemma}

\begin{proof}
It follows from \cite[Lemma 1.3]{Win0}.
\end{proof}

\begin{lemma}
\label{prelim-lm7}
Let $p \in (1,n)$. Then there exists $C>0$ such that
\begin{equation*}
    \| \nabla v(\cdot,t) \|_{L^{\frac{np}{n-p}}(\Omega)} \leq C \|u(\cdot,t)\|_{L^p(\Omega)} \quad \text{\rm for all}\,\,  t\in (0,T_{\max}).
\end{equation*}
\end{lemma}

\begin{proof}
It follows from  the Gagliardo-Nirenberg-Sobolev Embedding Theorem, standard elliptic theory, and the second equation of \eqref{main-eq}.
\end{proof}

{Next, we  provide some basic estimates for the solutions of \eqref{main-eq}.

\begin{lemma}
\label{prelim-lm4}
  For any given $u_0$ satisfying \eqref{initial-cond-eq},  we have
\vspace{-0.05in}\begin{equation}
\label{L-1-estimate-eq1-1}
\int_\Omega u_{\sigma}(x,t;u_0)dx\le m_0^*(u_0):= \max\Big\{\int_\Omega u_0,\Big(\frac{a_{\sup}}{b_{\inf}}\Big)^{\frac{1}{1+\sigma}}|\Omega|\Big\}\quad \forall\, 0<t<T_{\max}^{\sigma}(u_0)
\end{equation}
and
\begin{equation}
\label{L-sigma-estimate-eq1}
\int_{\tau}^{t} \int_\Omega u_{\sigma}^{2+\sigma}(x,s;u_0)dxds\le  \big (1+a_{\sup}\cdot (t-\tau)\big ) m_0^*(u_0)\quad \forall\, 0<\tau<t<T_{\max}^{\sigma}(u_0).
\end{equation}
\end{lemma}

\begin{proof}
First, by integrating the first equation in \eqref{main-eq} with respect to $x$, we get that
\begin{align*}
\frac{d}{dt}\int_\Omega u_{\sigma} (x,t;u_0)dx&=\int_{\Omega}\Delta u_{\sigma}(x,t;u_0) dx- \chi  \int_{\Omega} \nabla \cdot \Big(\frac{u_{\sigma}(x,t;u_0)}{v_{\sigma}(x,t;u_0)}\nabla v_{\sigma}(x,t;u_0) \Big)dx\\
&\quad + \int_{\Omega} \Big\{a(x,t)u_{\sigma}(x,t;u_0) -b(x,t) u^{2+\sigma}_{\sigma}(x,t;u_0)\Big\}dx\\
    &\le a_{\sup}\int_\Omega u_{\sigma}(x,t;u_0)dx-b_{\inf} \int_\Omega u_{\sigma} ^{2+\sigma}(x,t;u_0)dx\\
&\le a_{\sup}\int_\Omega u_{\sigma}(x,t;u_0)dx-\frac{b_{\inf}}{|\Omega|^{1+\sigma}}\Big(\int_\Omega u_{\sigma}(x,t;u_0)dx\Big)^{2+\sigma}\quad \forall\, t\in (0,T_{\max}^{\sigma}(u_0)).
\end{align*}
Let $y(t)=\int_{\Omega} u_{\sigma}(x,t;u_0)dx$ for  $t\in (0,T^{\sigma}_{\max})$. Then
\begin{equation*}
    y'(t) \leq   a_{\sup} y(t)- \frac{b_{\inf}}{|\Omega|^{1+\sigma}} y^{2+\sigma}(t) \quad \text{for all}\,\,  0<t<T_{\max}.
\end{equation*}
The inequality \eqref{L-1-estimate-eq1-1} then follows from
 the comparison principle for scalar ODEs  and continuity of
$\int_\Omega u_\sigma(x,t;u_0)dx$ in $t\in [0,T_{\max}^\sigma(u_0))$.

Second, observe that,  for any $0<\tau<t<T_{\max}^{\sigma}(u_0)$, we have
\begin{align*}
& \int_\tau ^{t} \int_\Omega b(x,t)u_{\sigma}^{2+\sigma}(x,s;u_0)dxds\\
&=\int_\Omega u_{\sigma}(x,\tau;u_0)dx-\int_\Omega u_{\sigma}(x,t;u_0)dx+\int_\tau ^{t}\int_\Omega a(x,s)u_{\sigma}(x,s;u_0)dxds.
\end{align*}
This together with \eqref{L-1-estimate-eq1-1}  implies that  \eqref{L-sigma-estimate-eq1} holds.
\end{proof}}

\begin{lemma}
\label{prelim-lm5}
For every $t\in (0,T_{\max}^\sigma)$,
\begin{equation*}
    \int_{\Omega}\frac{|\nabla v_\sigma|^2}{v_\sigma^2} \leq \mu |\Omega|.
\end{equation*}
\end{lemma}

\begin{proof}
By multiplying the second equation in \eqref{main-eq} by $\frac{1}{v_\sigma}$, we obtain that
\begin{align*}
    0&= \int_{\Omega} \frac{1}{v_\sigma} \cdot \Big( \Delta v_\sigma -\mu v_\sigma + \nu u_\sigma\Big)\\
    &= \int_{\Omega}\frac{|\nabla v_\sigma|^2}{v_\sigma^2} - \mu |\Omega| + \nu \int_{\Omega}\frac{u_\sigma}{v_\sigma}
\end{align*}
for  all $t \in (0,T_{\max}^\sigma)$. This implies that
\begin{equation*}
    \int_{\Omega}\frac{|\nabla v_\sigma|^2}{v_\sigma^2} \leq \mu |\Omega|
\end{equation*}
for every $t\in (0,T_{\max}^\sigma)$.
The lemma thus follows.
\end{proof}

\begin{lemma}
\label{prelim-lm6}
For any $0<\tau<t<T_{\max}^\sigma$, we have
\begin{align}
\label{ln-u-estimate-eq2}
 \int_\Omega \ln u_{\sigma}(x,t)dx&\ge \int_\Omega \ln u_{\sigma}(x,\tau)dx+ \Big(- \frac{\mu |\Omega|\chi^2}{4} +a_{\inf}|\Omega|-
\frac{b_{\sup}}{2+\sigma}|\Omega|\Big)(t-\tau)\nonumber\\
&\quad -\frac{b_{\sup}(1+\sigma)}{2+\sigma}\Big(1+a_{\sup}(t-\tau)\Big)m_0^*(u_0).
\end{align}
\end{lemma}

\begin{proof}
It follows from a little modified arguments of \cite[Lemma 3.2]{FuWiYo1}.  For the completeness, we provide a proof in the following.

{ Multiplying the first equation in \eqref{main-eq} by $\frac{1}{u_{\sigma}}$ and then  integrating over $\Omega$, we obtain
\begin{equation*}
    \frac{d}{dt} \int_{\Omega} \ln{u_\sigma}\ge \int_{\Omega} \frac{|\nabla u_\sigma|^2}{u_\sigma^2} - { \chi \int_{\Omega} \frac{\nabla u_\sigma}{u_\sigma}\cdot \frac{\nabla v_\sigma}{v_\sigma}}  +a_{\inf}|\Omega|-b_{\sup}\int_{\Omega}u_\sigma^{1+\sigma} \quad \text{for all}\,\,  t\in (0,T_{\max}^\sigma(u_0)).
\end{equation*}
By Young's inequality and Lemma \ref{prelim-lm5}, we have
\begin{equation*}
\chi \int_{\Omega} \frac{\nabla u_\sigma}{u_\sigma}\cdot \frac{\nabla v_\sigma}{v_\sigma} \leq \int_{\Omega} \frac{|\nabla u_\sigma|^2}{u_\sigma^2}+ \frac{\chi^2}{4} \int_{\Omega}\frac{|\nabla v_\sigma|^2}{v_\sigma^2} \leq \int_{\Omega} \frac{|\nabla u_\sigma|^2}{u_\sigma^2}+ \frac{\mu |\Omega|\chi^2}{4}.
\end{equation*}
Hence
\begin{equation*}
    \frac{d}{dt} \int_{\Omega} \ln{u_\sigma}\ge - \frac{\mu |\Omega|\chi^2}{4} +a_{\inf}|\Omega|-
b_{\sup}\int_\Omega u_{\sigma}^{1+\sigma}dx\quad \forall\,   t\in (0,T_{\max}).
\end{equation*}
By Young's inequality again,
\begin{equation*}
    \int_\Omega u_{\sigma}^{1+\sigma}\le \frac{1+\sigma}{2+\sigma}\int u_{\sigma}^{2+\sigma}+\frac{1}{2+\sigma}|\Omega|.
\end{equation*}
Hence
\begin{align*}
    \frac{d}{dt} \int_{\Omega} \ln{u_\sigma} &\ge - \frac{\mu |\Omega|\chi^2}{4} +a_{\inf}|\Omega|-
b_{\sup}\frac{1+\sigma}{2+\sigma}\int_\Omega u_{\sigma}^{2+\sigma} -\frac{b_{\sup}}{2+\sigma}|\Omega|.
\end{align*}
This together with  \eqref{L-sigma-estimate-eq1} implies   \eqref{ln-u-estimate-eq2}.}
\end{proof}

{
\begin{lemma}
    \label{prelim-lm8}
For any $\varepsilon>0$, there is $C(\varepsilon,p)>0$ such that
\begin{equation*}
\int_\Omega v_\sigma^{p+1}\le \varepsilon \int_\Omega u_\sigma^{p+1}+ C(\varepsilon,p)\Big(\int_\Omega u_\sigma\Big)^{p+1}\quad \text{for all} \;\; t\in (0,T_{\max}^\sigma).
\end{equation*}
\end{lemma}

\begin{proof}
First, multiplying the second equation in \eqref{main-eq} by $v_\sigma^{p}$ and integrating over $\Omega$, along with applying Young's inequality, we have
\begin{equation*}
    p\int_{\Omega} v_\sigma^{p-1}|\nabla v_\sigma|^2 + \mu \int_{\Omega}v_\sigma^{p+1} = \nu \int_{\Omega}u_\sigma v_\sigma^{p} \leq  \left(\frac{\nu}{p+1}\right)^{p+1}\left(\frac{p}{\mu}\right)^p\int_{\Omega} u_\sigma^{p+1} +\mu \int_{\Omega} v_\sigma^{p+1},
\end{equation*}
which implies
\begin{equation}
\label{aux-new-eq1*}
    \int_{\Omega} v_\sigma^{p-1}|\nabla v_\sigma|^2  \leq \frac{1}{p} \left(\frac{\nu}{p+1}\right)^{p+1}\left(\frac{p}{\mu}\right)^p\int_{\Omega} u_\sigma^{p+1} \quad\text{for all}\;\; t \in (0,T_{\max}^\sigma).
\end{equation}
Then, for any  $\tilde \varepsilon>0,$ there are $ C_1,\tilde C_1>0$ such that
 \begin{align*}
  &  \int_{\Omega} v_\sigma^{p+1}=\int_{\Omega} { (v_\sigma^{\frac{p+1}{2}})^2}\\
&\le \tilde\varepsilon \int_{\Omega} |\nabla v_\sigma^{\frac{p+1}{2}}|^2+C _1 \left(\int_{\Omega}  v_\sigma^{\frac{p+1}{2}} \right)^{2}\quad \quad \quad \quad \quad \quad\quad \quad \quad \quad\qquad\qquad\quad\,\,\, \, \;\;\text{(by Ehrling's  lemma)}\nonumber\\
   &\le \frac{\tilde\varepsilon (p+1)^2}{4}\int_{\Omega} v_\sigma^{p-1}|\nabla v_\sigma|^2 +
  C_1\Big( \int_\Omega v_\sigma^{p}\Big)\Big(\int_\Omega v_\sigma\Big) \quad \quad \quad \quad \quad \quad\quad \quad \quad\quad\text{(by H\"older's inequality)}\nonumber\\
&\le \frac{\tilde\varepsilon (p+1)^2}{4p}\left(\frac{\nu}{p+1}\right)^{p+1}\left(\frac{p}{\mu}\right)^p\int_{\Omega} u_\sigma^{p+1} +
  C_1\Big( \int_\Omega v_\sigma^{p}\Big)\Big(\int_\Omega v_\sigma\Big) \quad \quad \quad \quad \,\,\, \text{(by \eqref{aux-new-eq1*})}\nonumber\\
   &\le  \frac{\tilde\varepsilon (p+1)^2}{4p}\left(\frac{\nu}{p+1}\right)^{p+1}\left(\frac{p}{\mu}\right)^p\int_{\Omega} u_\sigma^{p+1} +
   \frac{1}{2} \int_\Omega v_\sigma^{p+1}+\tilde C_1\Big(\int_\Omega v_\sigma\Big)^{p+1} \quad \;\text{(by Young's inequality)}
\end{align*}
for all $t \in (0, T_{\rm max}^\sigma)$. The lemma then follows with the fact  $\int_\Omega v_\sigma=\frac{\nu}{\mu}\int_\Omega u_\sigma$ and taking $C=\tilde C_1(\frac{\nu}{\mu})^{p+1}$ and $\tilde \varepsilon =\frac{2p}{(p+1)^2}(\frac{p+1}{\nu})^{p+1}(\frac{\mu}{p})^p \varepsilon.$
\end{proof}}

\medskip

\begin{lemma} [Reverse H\"older's inequality]

Assume that $p\in(1,\infty)$. Then, for any given  Lebesgue  measurable  functions $f$ and $g$ on $\Omega$ with $g(x)\not\equiv 0$ for almost all $x\in \Omega$,
\begin{equation*}
    \|fg\|_{L^1(\Omega)} \ge \|f\|_{{L^{\frac{1}{p}}(\Omega)}}\|g\|_{{L^{\frac{-1}{p-1}}(\Omega)}}.
\end{equation*}
\end{lemma}
\begin{proof}
An application of H\"older's inequality gives
\begin{align*}
    \Big\||f|^{\frac{1}{p}}\Big\|_{L^1(\Omega)}&=\Big\||fg|^{\frac{1}{p}}|g|^{\frac{-1}{p}}\Big\|_{L^1(\Omega)}\\
    & \leq \Big\||fg|^{\frac{1}{p}}\Big\|_{L^p(\Omega)} \Big\||g|^{\frac{-1}{p}}\Big\|_{L^{\frac{p}{p-1}}(\Omega)}\\
    &=\|fg\|_{L^1(\Omega)}^{\frac{1}{p}} \Big\||g|^{\frac{-1}{p-1}}\Big\|_{L^1(\Omega)}^{\frac{p-1}{p}}.
    \end{align*}
This implies that
\begin{equation*}
    \Big\||f|^{\frac{1}{p}}\Big\|_{L^1(\Omega)}^{p} \leq \|fg\|_{L^1(\Omega)} \Big\||g|^{\frac{-1}{p-1}}\Big\|_{L^1(\Omega)}^{p-1}.
\end{equation*}
Therefore,
\begin{equation*}
    \|fg\|_{L^1(\Omega)} \ge \Big\||f|^{\frac{1}{p}}\Big\|_{L^1(\Omega)}^{p} \Big\||g|^{\frac{-1}{p-1}}\Big\|_{L^1(\Omega)}^{-(p-1)}.
\end{equation*}
The lemma follows.
\end{proof}

\section{Mass persistence}

In this section, we investigate the mass persistence of positive solutions of \eqref{main-eq} and prove the Proposition \ref{mass-persistence-prop}.

Throughout this section, we assume that $(u_\sigma(x,t),v_\sigma(x,t))=(u_\sigma(x,t;u_0),v_\sigma(x,t;u_0))$ is the unique classical  solution of \eqref{main-eq} on the maximal interval $(0,T^{\sigma}_{\max}):=(0,T^{\sigma}_{\max}(u_0))$ satisfying
the initial condition $u_\sigma(x,0;u_0)=u_0(x)$.  We may drop $(x,t)$ in $(u_\sigma(x,t),v_\sigma(x,t))$ if no confusion occurs.

We  first present some lemmas.

\begin{lemma}
\label{mass-lem-3-1}
Let $R>0$ be such that
\begin{equation}
    \label{def-R-mu-chi}
    R>
\begin{cases}
\frac{\mu \chi^2}{4}, &\text{if\,\, $0< \chi \leq 2$}\\
\mu(\chi-1), &\text{if \,$\chi>2$.}\\
\end{cases}
\end{equation}
Then there is $\beta>0$, $\beta\not =\chi$ such that
\begin{equation}
    \label{def-R}
    \frac{(p+1)\beta \mu}{p}-R<0,
\end{equation}
where $p$ is given by
\begin{equation}
    \label{def-p}
    p=\frac{4\beta}{( \chi-\beta)^2}.
\end{equation}
\end{lemma}

\begin{proof}
First, by \eqref{def-p}, \eqref{def-R} is equivalent to
\begin{equation*}
    \frac{\left(\frac{4\beta}{( \chi-\beta)^2}+1\right)\beta\mu}{\frac{4\beta}{( \chi-\beta)^2}}-R<0,
\end{equation*}
which is equivalent to
\begin{equation*}
    \mu \beta^2+2\mu(2-\chi)\beta+\mu \chi^2-4R<0.
\end{equation*}
Let
\begin{equation}
\label{quad-func-f}
    f(\beta)=\mu \beta^2+2\mu(2-\chi)\beta+\mu \chi^2-4R.
\end{equation}
Then $f(\beta)=0$ if and only if $\beta=\beta_-$ or $\beta=\beta_+$, where
\begin{equation}
\label{root-a-+}
    \beta_\pm=\chi-2\pm 2\sqrt{\frac{R}{\mu}+1-\chi}.
\end{equation}

Next,  if $\chi>2$, then  $R>\mu(\chi-1)$. This implies that  $\beta_\pm$ are real numbers, $f(\beta)<0$ for $\beta\in (\beta_-,\beta_+)$, and
\begin{equation*}
    \beta_+=\chi-2+2\sqrt{\frac{R}{\mu}+1-\chi}>0.
\end{equation*}
Therefore, there is $\beta\in (0,\beta_+)$ with $\beta\not =\chi$ such that \eqref{def-R} holds.

Now, if $\chi\le 2$, then  $R>\frac{\mu \chi^2}{4}$. This
implies that
$$
{\frac{R}{\mu}+1-\chi}>\frac{\mu \chi^2}{4}\frac{1}{\mu}+1-\chi=\frac{(\chi-2)^2}{4}\ge 0,
$$
and
\begin{equation*}
    \beta_+=\chi-2+2\sqrt{\frac{R}{\mu}+1-\chi}>\chi-2+2\sqrt{\frac{(2-\chi)^2}{4}}=0.
\end{equation*}
Therefore, there is also $\beta\in (0,\beta_+)$ with $\beta\not =\chi$ such that \eqref{def-R} holds.
\end{proof}

\begin{lemma}
\label{mass-lem-3-2}
Let $p>0$. Then,
\begin{equation}
\label{eq-3-6}
p \int_{\Omega} \frac{u_\sigma^{-p-1}}{v_\sigma} \nabla u_\sigma \cdot \nabla v_\sigma \leq \mu \int_{\Omega} u_\sigma^{-p} - \int_{\Omega}u_\sigma^{-p} \frac{|\nabla v_\sigma|^2}{v_\sigma^2} \quad \text{for all}\,\, t\in (0,T_{\max}).
\end{equation}
\end{lemma}
\begin{proof}
In this proof,  we drop $\sigma$ in $(u_\sigma,v_\sigma)$ and $T_{\max}^\sigma$  if no confusion occurs.

Multiplying the second equation in \eqref{main-eq} by $\frac{u^{-p}}{v}$ and then integrating over $\Omega$ with respect to $x$, we obtain that
\begin{align*}
    0&= \int_{\Omega} \frac{u^{-p}}{v} \cdot \big( \Delta v - \mu v+ \nu u\big)\\
    &= - \int_{\Omega}\frac{(-p) u^{-p-1}v \nabla u -u^{-p} \nabla v}{v^2} \cdot \nabla v - \mu \int_{\Omega} u^{-p} + \nu   \int_{\Omega}\frac{u^{-p+1}}{v} \quad \text{for all}\,\, t\in (0,T_{\max}).
\end{align*}
Thus we have,
\begin{equation*}
     p \int_{\Omega} \frac{u^{-p-1}}{v} \nabla u \cdot \nabla v + \int_{\Omega} u^{-p} \frac{|\nabla v|^2}{v^2} + \nu  \int_{\Omega}\frac{u^{-p+1}}{v}= \mu \int_{\Omega} u^{-p} \quad \text{for all}\,\, t\in (0,T_{\max}).
     \vspace{0.1in}
\end{equation*}
This together with the nonnegativity of  $\int_{\Omega}\frac{u^{-p+1}}{v}$ implies \eqref{eq-3-6}.
\end{proof}

\begin{lemma}
\label{mass-lem-3-3}
Let $p>0$. Then for every $\beta>0$, we have
\begin{align}
\label{neg-nabla-u-v}
     (p+1)\chi \int_{\Omega}  \frac{u_\sigma^{-p-1}}{v_\sigma} \nabla u_\sigma \cdot \nabla v_\sigma &\leq (p+1) \int_{\Omega} u_\sigma^{-p-2}|\nabla u_\sigma|^2 + \frac{(p+1)\beta\mu}{p} \int_{\Omega} u_\sigma^{-p} \nonumber\\
     & \quad + \left[\frac{ (p+1)(\chi-\beta)^2 }{4}-\frac{(p+1)\beta}{p} \right] \int_{\Omega} u_\sigma^{-p} \frac{|\nabla v_\sigma|^2}{v_\sigma^2}
\end{align}
for all $t\in (0,T_{\max}^\sigma)$.
\end{lemma}
\begin{proof}
In this proof, we also drop $\sigma$ in $(u_\sigma,v_\sigma)$ and $T_{\max}^\sigma$ if no confusion occurs.

Note that, for every $\beta>0$,
\begin{equation*}
    \chi \int_{\Omega}  \frac{u^{-p-1}}{v} \nabla u \cdot \nabla v = \underbrace{(\chi-\beta) \int_{\Omega}  \frac{u^{-p-1}}{v} \nabla u \cdot \nabla v}_{I_1} + \underbrace{\beta \int_{\Omega}  \frac{u^{-p-1}}{v} \nabla u \cdot \nabla v}_{I_2}.
\end{equation*}
By Young's inequality, we have
\begin{equation}
    \label{eq-I-1}
    I_1=(\chi-\beta) \int_{\Omega}  \frac{u^{-p-1}}{v} \nabla u \cdot \nabla v \leq \int_{\Omega} u^{-p-2} |\nabla u|^2 + \frac{ (\chi-\beta)^2 }{4} \int_{\Omega} u^{-p} \frac{|\nabla v|^2}{v^2}.
\end{equation}
By Lemma \ref{mass-lem-3-2}, we get
\begin{equation}
    \label{eq-I-2}
    I_2=\beta \int_{\Omega}  \frac{u^{-p-1}}{v} \nabla u \cdot \nabla v \leq \frac{\beta\mu}{p} \int_{\Omega} u^{-p} - \frac{\beta}{p} \int_{\Omega} u^{-p} \frac{|\nabla v|^2}{v^2}.
\end{equation}
Therefore \eqref{eq-I-1} and \eqref{eq-I-2} entail the desired inequality \eqref{neg-nabla-u-v}.
\end{proof}

\begin{lemma}
\label{mass-lem-3-4}
Assume that \eqref{def-ainf-mu-chi} holds.  There is $p>0$ such that for every $\tau>0$,  there exists $C>0$ such that
\begin{equation*}
    \int_{\Omega} u_\sigma^{-p} \leq C \quad \text{for all}\,\, \tau<t< T_{\max}^\sigma.
\end{equation*}
\end{lemma}

\begin{proof}
It can be proved by the similar arguments as those in Lemma 6.2 in \cite{FuWiYo1}. For the self-completeness, we provide
a proof in the following.

{First of all,  fix $\chi,\mu>0$ and assume \eqref{def-ainf-mu-chi}.
Then there is  $\alpha>1$  such that
\begin{equation*}
    a_{\inf}> a_{\chi,\mu,\alpha}:=\begin{cases}
\frac{\alpha \mu \chi^2}{4}, &\text{if\,\, $0< \chi \leq 2,$}\\
\alpha \mu(\chi-1), &\text{if \,$\chi>2$.}
\end{cases}
\end{equation*}
By the arguments of Lemma \ref{mass-lem-3-1}, there is  $\beta_\alpha=\beta_\alpha(\chi,\alpha)>0$, $\beta_\alpha \not =\chi$, such that
\begin{equation*}
    \frac{(p_\alpha+1)\beta_\alpha  \mu}{p_\alpha }-a_{\chi,\mu,\alpha}=0,
\end{equation*}
where
\begin{equation*}
    p_\alpha =p_\alpha (\chi,\alpha ):=\frac{4\beta_\alpha }{( \chi-\beta_\alpha )^2}.
\end{equation*}

Next, multiplying the first equation in \eqref{main-eq} by $u^{-p_\alpha-1}$ and integrating over $\Omega$, we have that
\begin{align*}
 &   \frac{1}{p_\alpha} \cdot \frac{d}{dt} \int_{\Omega} u_{\sigma}^{-p_\alpha}\\
 &=-\int_{\Omega} u^{-p_\alpha-1}\Delta u_{\sigma}+\chi \int_{\Omega} u_{\sigma}^{-p_\alpha-1} \nabla \cdot \Big(\frac{u_{\sigma}}{v_{\sigma}} \nabla v_{\sigma}\Big)-  \int_{\Omega} a(\cdot,t) u_{\sigma}^{-p_\alpha} +  \int_{\Omega} b(\cdot,t)  u_{\sigma}^{-p_\alpha+1+\sigma}\\
    &\leq -(p_\alpha+1)\int_{\Omega} u_{\sigma}^{-p_\alpha-2}|\nabla u_{\sigma}|^2 + (p_\alpha+1) \chi \int_{\Omega} \frac{u_{\sigma}^{-p_\alpha-1}}{v_{\sigma}}\nabla u_{\sigma} \cdot \nabla v_{\sigma}\\
    &\quad -  a_{\inf} \int_{\Omega} u_{\sigma}^{-p_\alpha} +  b_{\sup} \int_{\Omega} u_{\sigma}^{-p_\alpha+1+\sigma} \quad \text{for all}\;\; t\in (0,T_{\max}^{\sigma}(u_0)).
\end{align*}
By Lemma \ref{mass-lem-3-3},  we have that
\begin{align*}
    \frac{1}{p_\alpha} \cdot \frac{d}{dt} \int_{\Omega} u_{\sigma}^{-p_\alpha} &\leq (p_\alpha+1) \left[\frac{ (\chi-\beta_\alpha)^2 }{4}-\frac{\beta_\alpha}{p_\alpha} \right] \int_{\Omega} u_{\sigma}^{-p_\alpha} \frac{|\nabla v|^2}{v^2} \\
     &\quad + \left[\frac{(p_\alpha+1)\beta_\alpha\mu}{p_\alpha}-a_{\inf}\right] \int_{\Omega} u_{\sigma}^{-p_\alpha}+ b_{\sup} \int_{\Omega} u_{\sigma}^{-p_\alpha+1+\sigma} \quad \text{for all}\;\; t\in (0,T_{\max}^{\sigma}(u_0)).
\end{align*}
Now, notice  that we have $\frac{ (\chi-\beta_\alpha)^2 }{4}-\frac{\beta_\alpha}{p_\alpha}=0$
and $\frac{(p_\alpha+1)\beta_\alpha \mu}{p_\alpha}-a_{\inf}<0$. It then follows that
\begin{align*}
    \frac{1}{p_\alpha} \cdot \frac{d}{dt} \int_{\Omega} u_{\sigma}^{-p_\alpha} &\leq -(a_{\inf}-a_{\chi,\mu,\alpha}) \int_{\Omega} u_{\sigma}^{-p_\alpha} + b_{\sup} \int_{\Omega} u_{\sigma}^{-p_\alpha+1+\sigma}\quad \text{for all}\;\; t\in (0,T_{\max}^{\sigma}(u_0)).
\end{align*}
By Lemma \ref{prelim-lm4}, Young's and H\"older inequalities, we obtain that
\begin{align*}
& b_{\sup} \int_{\Omega} u^{-p_\alpha +1+\sigma}\le \begin{cases}
 \frac{a_{\inf}-a_{\chi,\mu,\alpha}}{2} \int_{\Omega} u^{-p_\alpha} + b_{\sup}^{\frac{p_\alpha}{1+\sigma}}\Big(\frac{2(p_\alpha-1-\sigma)}{p_\alpha (a_{\inf}-a_{\chi,\mu,\alpha})}\Big)^{\frac{p_\alpha-1-\sigma}{1+\sigma}}\frac{1+\sigma}{p_\alpha}|\Omega|,  \quad &{\rm if}\,\,\, p_\alpha>1+\sigma,\cr
 b_{\sup}|\Omega|, \quad &{\rm if}\,\,\, p_\alpha=1+\sigma,\cr
 b_{\sup}|\Omega|^{p_\alpha -\sigma}(m_0^*(u_0))^{1+\sigma-p_\alpha },
\quad &{\rm if}\,\,\, \sigma< p_\alpha <1+\sigma,\cr
 \frac{1+\sigma-p_\alpha}{2+\sigma}\int_\Omega u_{\sigma}^{2+\sigma}+  b_{\sup}^{\frac{2+\sigma}{p_\alpha+1}} \frac{p_\alpha +1}{2+\sigma}|\Omega|,\quad & {\rm if}\,\,\, p_\alpha\le \sigma,
\end{cases}
\end{align*}
for any $t\in (0,T_{\max}^\sigma(u_0))$,
where $m_0^*(u_0)$ is as in \eqref{L-1-estimate-eq1-1}.  It then follows that
\begin{equation*}
    \frac{d}{dt} \int_{\Omega} u_{\sigma}^{-p_\alpha} \leq -\frac{p_\alpha (a_{\inf}-a_{\chi,\mu,\alpha})}{2}\int_{\Omega} u_{\sigma} ^{-p_\alpha } + p_\alpha \int_\Omega u_{\sigma}^{2+\sigma}+ p_\alpha M_1^* \quad \text{for all}\;\; t\in (0,T_{\max}^{\sigma}(u_0)),
\end{equation*}
where
\begin{align*}
M_1^*&=\begin{cases}
b_{\sup}^{\frac{p_\alpha}{1+\sigma}}\Big(\frac{2(p_\alpha-1-\sigma)}{p_\alpha (a_{\inf}-a_{\chi,\mu,\alpha})}\Big)^{\frac{p_\alpha -1-\sigma}{1+\sigma}}\frac{1+\sigma}{p_\alpha }|\Omega|,  \quad &{\rm if}\quad p_\alpha >1+\sigma,\cr
   b_{\sup}|\Omega|,\quad &{\rm if}\quad p_\alpha =1+\sigma,\cr
   b_{\sup}|\Omega|^{p_\alpha -\sigma}(m_0^*)^{1+\sigma-p_\alpha},  \quad &{\rm if}\quad \sigma<p_\alpha < 1+\sigma,\cr
b_{\sup}^{\frac{2+\sigma}{p_\alpha +1}} \frac{p_\alpha +1}{2+\sigma}|\Omega|, \quad &{\rm if}\quad p_\alpha \le \sigma.
\end{cases}
\end{align*}

Now, let
\begin{equation*}
    c^*:=p_\alpha (1+a_{\sup})m_0^*(u_0)+p_\alpha M_1^*.
\end{equation*}
By  Lemma  \ref{prelim-lm4} and \cite[Lemma 3.4]{StSuWi}, we have}
\begin{equation*}
    \int_\Omega u_\sigma ^{-p_\alpha}(x,t)dx \le M_2^*(\tau,u_0):=\max\Big\{ \int_\Omega u_\sigma^{-p_\alpha}(x,\tau)dx+c^*,\frac{2c^*}{p_\alpha (a_{\inf}-a_{\chi,\mu,\alpha})}+2 c^*\Big\},
\end{equation*}
for all $0<\tau<t<T_{\max}^\sigma(u_0)$.  The lemma thus follows.
\end{proof}

Now, we prove Proposition \ref{mass-persistence-prop}.

\begin{proof}[Proof of Proposition \ref{mass-persistence-prop}]
(1) First of all,  by Lemma \ref{prelim-lm1}, it suffices to prove that
$$
\inf_{0\le t<\min\{T,{ T^{\sigma}_{\max}}\}}\int_\Omega { u_{\sigma}(x,t)}dx>0.
$$
Fix a $\tau\in (0,{ T^{\sigma}_{\max}})$. It is clear that
$$\inf_{0\le t\le \tau}\int_\Omega { u_{\sigma}(x,t)}dt>0.
$$
It then suffices to prove that  there exist $C=C(T)>0$  such that
\begin{equation}
\label{new-new-eq1}
    \int_{\Omega} { u_{\sigma}(x,t)}dx \ge C(T) \quad \text{\rm for all}\,\, t\in (\tau,\hat T),
\end{equation}
where $\hat T=\min\{T,{ T^{\sigma}_{\max}}\}$.
Note that this inequality follows from the similar arguments as those in \cite[Corollary 3.3]{FuWiYo1}. For the reader's convenience, we provide a proof for this inequality in the following.

Note that $L:=\int_{\Omega} {\ln{ u_{\sigma}(x,\tau)}}dx$ is finite. By Lemma \ref{prelim-lm6},
\begin{align*}
 \int_\Omega \ln u_{\sigma}(x,t)dx&\ge \int_\Omega \ln u_{\sigma}(x,\tau)dx+ \Big(- \frac{\mu |\Omega|\chi^2}{4} +a_{\inf}|\Omega|-
\frac{b_{\sup}}{2+\sigma}|\Omega|\Big)(t-\tau)\nonumber\\
&\quad -\frac{b_{\sup}(1+\sigma)}{2+\sigma}\Big(1+a_{\sup}(t-\tau)\Big)m_0^*(u_0)
\end{align*}
for all $t\in (\tau,\hat T)$.
By Jensen's inequality,  {we have
\begin{equation*}
    \int_{\Omega} \ln{u_\sigma(x,t)}dx=|\Omega| \cdot \int_{\Omega} \ln{u_\sigma(x,t)}\frac{dx}{|\Omega|} \leq |\Omega| \cdot \ln{\Big( \int_{\Omega} u_\sigma(x,t) \frac{dx}{|\Omega|}  \Big)} \quad \text{for all}\,\, t\in (0, T_{\max}^\sigma(u_0)).
\end{equation*}}
It then follows that there is  $C(\tau,\hat T)>0$ such that
\begin{equation*}
    \int_{\Omega} { u_\sigma(x,t)} \ge |\Omega| \cdot \exp{\Big( \frac{1}{|\Omega|} \cdot \int_{\Omega} { \ln{u_\sigma(x,t)}} dx   \Big)} \ge |\Omega| \cdot e^{\frac{1}{|\Omega|}\cdot C(\tau,\hat T)} \quad \text{for all}\,\, t\in (\tau, \hat T).
\end{equation*}
\eqref{new-new-eq1} then follows. { Note that \eqref{ln-u-estimate-eq1-2} follows from \eqref{ln-u-estimate-eq1-1} and Lemma \ref{prelim-lm1}.}

\medskip

(2) In view of the Reverse H\"older inequality, taking $f=1$, { $g=u_\sigma$ and $p\rightarrow \frac{p_\alpha+1}{p_\alpha}>1,$ we have
\begin{equation*}
    \int_{\Omega} u_\sigma \ge |\Omega|^{\frac{p_\alpha+1}{p_\alpha}} \left(\int_{\Omega}u_\sigma^{-p_\alpha}\right)^{-\frac{1}{p_\alpha}} \quad \text{for all}\,\, t\in (\tau,T_{\max}^{\sigma}(u_0)).
\end{equation*}
This together with  Lemma \ref{prelim-lm1}, Lemma \ref{mass-lem-3-4}, and the continuity of $\int_\Omega u_\sigma(x,t)dx$ in $t\in [0,T_{\max}^\sigma(u_0))$  implies that \eqref{ln-u-estimate-eq1-1}-\eqref{ln-u-estimate-eq1-2} also hold} for $T=\infty$.
\end{proof}

\section{$L^p$-Boundedness}

In this section, we study the $L^p$-boundedness of solutions of \eqref{main-eq} and prove Proposition \ref{main-prop2} and Theorem \ref{main-thm1}.
Throughout this section, we assume that $u_0(x)$ satisfies \eqref{initial-cond-eq} and
{ $(u_\sigma(x,t),v_\sigma(x,t))=(u_\sigma(x,t;u_0),v_\sigma(x,t;u_0))$} is the unique classical  solution of \eqref{main-eq} on the maximal interval { $(0,T^{\sigma}_{\max}):=(0,T^{\sigma}_{\max}(u_0))$} with the initial function $u_0$. We may drop $\sigma$ in $(u_\sigma(x,t),v_\sigma(x,t))$ and $T_{\max}^\sigma$ if no confusion occurs.

Note that, multiplying the first equation in \eqref{main-eq} by $u^{p-1}$ and integrating  over $\Omega$ with respect to $x$, we have
\begin{align*}
    \frac{1}{p} \cdot \frac{d}{dt} \int_{\Omega} u^p &=\int_{\Omega} u^{p-1}\Delta u-\chi \int_{\Omega} u^{p-1} \nabla \cdot \Big(\frac{u}{v}\nabla v\Big)+  \int_{\Omega} a(\cdot,t) u^p -  \int_{\Omega} b(\cdot,t)  { u^{p+1+\sigma}}\\
    &=-(p-1)\int_{\Omega} u^{p-2}|\nabla u|^2 + (p-1) \chi \int_{\Omega} \frac{u^{p-1}}{v}\nabla u \cdot \nabla v +\int_{\Omega} a(\cdot,t) u^p -  \int_{\Omega} b(\cdot,t) { u^{p+1+\sigma}}
\end{align*}
for $t\in (0,T_{\max})$. To get the $L^p$-boundedness of $u$, the main ingredient is to get proper estimates for $\int_{\Omega}\frac{u^{p-1}}{v} \nabla u \cdot \nabla v$. Observe that, by Young's inequality, we have
\begin{align}
\label{aux-new-eq00}
   \chi \int_{\Omega} \frac{u^{p-1}}{v}\nabla u \cdot \nabla v
&\le \int_\Omega u^{p-2} |\nabla u|^2+\frac{\chi^2}{4} \int_\Omega u^p\frac{|\nabla v|^2}{v^2} \nonumber\\
&\le\int_\Omega u^{p-2}|\nabla u|^2+\frac{\chi^2}{4\inf_{x\in\Omega}v(x,t)}\int_\Omega u^p\frac{|\nabla v|^2}{v} \quad \qquad\qquad\qquad\qquad\qquad  \nonumber\\
&\le\int_\Omega u^{p-2}|\nabla u|^2+\frac{\chi^2}{4\inf_{x\in\Omega}v(x,t)} \Big(\int_\Omega u^{p+1}\Big)^{\frac{p}{p+1}}\Big(\int_\Omega \frac{|\nabla v|^{2p+2}}{v^{p+1}}\Big)^{\frac{1}{p+1}} \quad \quad  
\end{align}
for $t\in (0,T_{\max})$, where the last inequality follows from H\"older inequality. We will then provide some estimate for $\int_\Omega \frac{|\nabla v|^{2p}}{v^{k+1}}$ in subsection 4.1
 (see Proposition \ref{main-prop2}), which will enable us to provide some estimate for $\int_\Omega \Big(\frac{|\nabla v|}{v}\Big)^{2p+2}$
in terms of $\int_\Omega u^{p+1}$. Finally, using Proposition \ref{main-prop2}, we prove Theorem \ref{main-thm1} in subsection 4.2.

\subsection{Estimates for $\int_{\Omega} \frac{|\nabla v|^{2p}}{v^{k+1}} $ and proof of Proposition \ref{main-prop2}}

In this subsection, we provide some estimate for $\int_{\Omega} \frac{|\nabla v|^{2p}}{v^{k+1}}$ and prove Proposition \ref{main-prop2}.

We first prove some lemmas.

\begin{lemma}
\label{lem-4-5}
Let $p \ge 3$  and $k \ge 1.$  Then
\begin{equation*}
    \int_{\Omega}\frac{|\nabla v|^{2p}}{v^{k+1}} \leq \frac{(p-1)^2}{k^2} \int_{\Omega}\frac{|\nabla v|^{2p-6}}{v^{k-1}}\big |\nabla |\nabla v|^2\big|^2 + \frac{2\mu}{k} \int_{\Omega}\frac{|\nabla v|^{2p-2}}{v^{k-1}}
\end{equation*}
{for all $t \in (0, T_{\rm max})$}.
\end{lemma}

 \begin{proof}
 Multiplying the second equation in \eqref{main-eq} by $\frac{|\nabla v|^{2p-2}}{v^{k}}$ and then integrating over $\Omega$ with respect to $x$, we have that
\begin{align*}
    0&= \int_{\Omega} \frac{|\nabla v|^{2p-2}}{v^{k}} \cdot \big( \Delta v -\mu v+\nu u\big)\\
    &= - \int_{\Omega}\nabla v \cdot \Big[{(p-1)\frac{|\nabla v|^{2p-4}}{v^{k}} \nabla |\nabla v|^2 - k \frac{|\nabla v|^{2p-2}}{v^{k+1}}} \nabla v \Big]  - \mu \int_{\Omega}\frac{|\nabla v|^{2p-2}}{v^{k-1}} + \nu \int_{\Omega}\frac{u|\nabla v|^{2p-2}}{v^{k}}.
\end{align*}
Therefore
\begin{equation}
\label{eq-4-41}
   k \int_{\Omega} \frac{|\nabla v|^{2p}}{v^{k+1}} + \nu \int_{\Omega}\frac{u|\nabla v|^{2p-2}}{v^{k}} = (p-1) \int_{\Omega} \frac{|\nabla v|^{2p-4}}{v^{k}} \nabla v \cdot \nabla |\nabla v|^2 + \mu \int_{\Omega}\frac{|\nabla v|^{2p-2}}{v^{k-1}}
\end{equation}
for all $t\in (0,T_{\max})$.  By Young's inequality, we have
\begin{equation*}
    (p-1) \int_{\Omega} \frac{|\nabla v|^{2p-4}}{v^{k}} \nabla v \cdot \nabla |\nabla v|^2 \leq \frac{k}{2}\int_{\Omega} \frac{|\nabla v|^{2p}}{v^{k+1}} + \frac{(p-1)^2}{2k}  \int_{\Omega} \frac{|\nabla v|^{2p-6}}{v^{k-1}} \big|\nabla |\nabla v|^2\big|^2
\end{equation*}
for all $t\in (0,T_{\max})$.  This together with \eqref{eq-4-41} implies that
$$
\frac{k}{2}\int_{\Omega} \frac{|\nabla v|^{2p}}{v^{k+1}} \le  \frac{(p-1)^2}{2k}  \int_{\Omega} \frac{|\nabla v|^{2p-6}}{v^{k-1}} \big|\nabla |\nabla v|^2\big|^2 + \mu \int_{\Omega}\frac{|\nabla v|^{2p-2}}{v^{k-1}}.
$$
The lemma then follows by multiplying  the above inequality both sides by $\frac{2}{k}$.
\end{proof}

To prove Proposition \ref{main-prop2},  we then provide some estimate for
 $\int_{\Omega}\frac{|\nabla v|^{2p-6}}{v^{k-1}}\big |\nabla |\nabla v|^2\big|^2$. Note  that
 \begin{align*}
    \int_{\Omega} \frac{|\nabla v|^{2p-6}}{v^{k-1}}\big |\nabla |\nabla v|^2\big|^2 &= \int_{\partial \Omega} \frac{|\nabla v|^{2p-4}}{v^{k-1}} \frac{\partial |\nabla v|^2}{\partial \nu} + (k-1)\int_{\Omega} \frac{ |\nabla v|^{2p-4}}{v^{k}} \nabla v \cdot \nabla |\nabla v|^2 \nonumber\\
    & \quad -  \int_{\Omega} \frac{|\nabla v|^{2p-4}}{v^{k-1}}  \Delta |\nabla v|^2- (p-3) \int_{\Omega} \frac{|\nabla v|^{2p-6}}{v^{k-1}}\big |\nabla |\nabla v|^2\big|^2.
\end{align*}
Hence
\begin{align}
\label{eq-4-43}
  (p-2)  \int_{\Omega} \frac{|\nabla v|^{2p-6}}{v^{k-1}}\big |\nabla |\nabla v|^2\big|^2 &= \int_{\partial \Omega} \frac{|\nabla v|^{2p-4}}{v^{k-1}} \frac{\partial |\nabla v|^2}{\partial \nu} + (k-1)\int_{\Omega} \frac{ |\nabla v|^{2p-4}}{v^{k}} \nabla v \cdot \nabla |\nabla v|^2 \nonumber\\
    & \quad -  \int_{\Omega} \frac{|\nabla v|^{2p-4}}{v^{k-1}}  \Delta |\nabla v|^2.
\end{align}
Note also that
\begin{equation*}
    \Delta |\nabla v|^2=2 \nabla v \cdot \nabla(\Delta v)+{ 2|D^2 v|^2}= 2 \nabla v \cdot \nabla(\mu v-\nu u)+{ 2|D^2 v|^2}.
\end{equation*}
Hence
\begin{align}
    \label{eq-4-45}
   & -  \int_{\Omega} \frac{|\nabla v|^{2p-4}}{v^{k-1}}  \Delta |\nabla v|^2 \nonumber\\
    &=-2 \mu \int_{\Omega} \frac{|\nabla v|^{2p-2}}{v^{k-1}}+2 \nu \int_{\Omega} \frac{|\nabla v|^{2p-4}}{v^{k-1}} \nabla u \cdot \nabla v -2  \int_{\Omega} \frac{|\nabla v|^{2p-4}}{v^{k-1}} |D^2v|^2.
\end{align}
To provide some estimate for
 $\int_{\Omega}\frac{|\nabla v|^{2p-6}}{v^{k-1}}\big |\nabla |\nabla v|^2\big|^2$,   we   provide some estimates     for $\int_{\partial \Omega} \frac{|\nabla v|^{2p-4}}{v^{k-1}} \frac{\partial |\nabla v|^2}{\partial \nu}$
 and $ \int_{\Omega}\frac{|\nabla v|^{2p-4}}{v^{k-1}} \nabla u \cdot \nabla v$ in next two lemmas.

\begin{lemma}
\label{lem-4-4}
For every $\epsilon>0$, there is $C_1=C_1(p,k,\epsilon)>0$ such that
\begin{equation*}
     \int_{\partial \Omega} \frac{|\nabla v|^{2p-4}}{v^{k-1}} \frac{\partial |\nabla v|^2}{\partial \nu} \leq \epsilon \int_{ \Omega} \frac{|\nabla v|^{2p-6}}{v^{k-1}}\big |\nabla |\nabla v|^2\big|^2 + C_1 \int_{ \Omega} \frac{|\nabla v|^{2p-2}}{v^{k-1}}
\end{equation*}
\end{lemma}

\begin{proof}
First, by \cite[Lemma 4.2]{MiSo}, there is $K_1=K_1(\Omega)>0$ such that  for every $w\in C^2(\bar\Omega)$ satisfying $\frac{\partial w}{\partial \nu}=0$ on $\partial \Omega$,
\begin{equation}
\label{new-eq2}
    \frac{\partial |\nabla w|^2}{\partial \nu} \leq K_1 |\nabla w|^2 \quad \text{on} \quad \partial \Omega.
\end{equation}
By the trace theorem (see \cite[Theorem 3.37]{Mcl}), for every $\frac{1}{2}<\theta<1$, there is $K_2$ such that
\begin{equation*}
    \|w\|_{L^2(\p \Omega)}\le K_2 \|w\|_{W^{\theta,2}(\Omega)}\quad \forall\, w\in W^{\theta,2}(\Omega).
\end{equation*}
By \cite[Theorem 11.6]{Ama}, the following interpolation holds.
\begin{equation*}
    (L^2(\Omega),W^{1,2}(\Omega))_{\theta,2}=W^{\theta,2}(\Omega).
\end{equation*}
By \cite[Theorem 1.3.3]{Tri}, there is $K_3$ such that
\begin{equation*}
    \|w\|_{W^{\theta,2}(\Omega)}\le K_3 \|w\|_{L^2(\Omega)}^{1-\theta}\|w\|_{W^{1,2}(\Omega)}^\theta.
\end{equation*}
This together with Young's inequality and $\|w\|_{W^{1,2}(\Omega)}\le \|w\|_{L^2(\Omega)}+\|\nabla w\|_{L^2(\Omega)}$ implies that, for every $\epsilon>0$, there is $C(\epsilon)>0$ such that
\begin{equation}
\label{new-eq5}
    \|w\|_{L^2(\partial\Omega)} \leq \epsilon \|\nabla w\|_{L^2(\Omega)}+C(\epsilon)\|w\|_{L^2(\Omega)}, \quad \forall w \in W^{1,2}(\Omega).
\end{equation}

Next, by \eqref{new-eq2} and \eqref{new-eq5}, for every $\epsilon>0$, there is  $C(p,\epsilon,K_1)>0$ such that
\begin{align}
\label{new-eq6}
    \int_{\partial \Omega} \frac{|\nabla v|^{2p-4}}{v^{k-1}} \frac{\partial |\nabla v|^2}{\partial \nu} &\leq K_1(\Omega) \int_{\partial \Omega} \frac{|\nabla v|^{2p-2}}{v^{k-1}}\nonumber\\
    & \leq \frac{4\epsilon}{(p-1)^2} \int_{ \Omega}  \left|\nabla \left( \frac{|\nabla v|^{p-1}}{v^{\frac{k-1}{2}}}\right)\right|^2 + C(p,k,\epsilon,K_1) \int_{ \Omega} \frac{|\nabla v|^{2p-2}}{v^{k-1}}.
\end{align}
Note that
\begin{align}
\label{new-eq7}
     \int_{ \Omega}  \left|\nabla \left( \frac{|\nabla v|^{p-1}}{v^{\frac{k-1}{2}}}\right)\right|^2 &= \int_{ \Omega}  \Bigg| \frac{p-1}{2} \frac{|\nabla v|^{p-3}}{v^{\frac{k-1}{2}}} \nabla |\nabla v|^2 - \frac{k-1}{2} \frac{|\nabla v|^{p-1}}{v^{\frac{k+1}{2}}}  \nabla v \Bigg|^2\nonumber\\
     &=\frac{(p-1)^2}{4} \int_{ \Omega} \frac{|\nabla v|^{2p-6}}{v^{k-1}}\big |\nabla |\nabla v|^2\big|^2 + \frac{(k-1)^2}{4} \int_{\Omega} \frac{|\nabla v|^{2p}}{v^{k+1}}\nonumber\\
     &\quad -\frac{(p-1)(k-1)}{2} \int_{\Omega} \frac{ |\nabla v|^{2p-4}}{v^{k}} \nabla v \cdot \nabla |\nabla v|^2.
\end{align}
By \eqref{eq-4-41},
\begin{equation*}
   \frac{(k-1)^2}{4} \int_{\Omega} \frac{|\nabla v|^{2p}}{v^{k+1}} \leq  \frac{(k-1)^2(p-1)}{4k} \int_{\Omega} \frac{|\nabla v|^{2p-4}}{v^{k}} \nabla v \cdot \nabla |\nabla v|^2 + \frac{\mu (k-1)^2 }{4k} \int_{\Omega}\frac{|\nabla v|^{2p-2}}{v^{k-1}}.
\end{equation*}
This together with \eqref{new-eq7} implies that
\begin{align}
\label{new-eq8}
     \int_{ \Omega}  \left|\nabla \left( \frac{|\nabla v|^{p-1}}{v^{\frac{k-1}{2}}}\right)\right|^2 &\le  \frac{(p-1)^2}{4} \int_{ \Omega} \frac{|\nabla v|^{2p-6}}{v^{k-1}}\big |\nabla |\nabla v|^2\big|^2+ \frac{\mu (k-1)^2 }{4k} \int_{\Omega}\frac{|\nabla v|^{2p-2}}{v^{k-1}}\nonumber\\
     &\quad + \frac{(p-1)(k-1)}{2} \left(\frac{k-1}{2k}-1\right) \int_{\Omega} \frac{ |\nabla v|^{2p-4}}{v^{k}} \nabla v \cdot \nabla |\nabla v|^2\nonumber\\
     &= \frac{(p-1)^2}{4} \int_{ \Omega} \frac{|\nabla v|^{2p-6}}{v^{k-1}}\big |\nabla |\nabla v|^2\big|^2+ \frac{\mu (k-1)^2 }{4k} \int_{\Omega}\frac{|\nabla v|^{2p-2}}{v^{k-1}}\nonumber\\
     &\quad -\frac{(p-1)(k-1)(k+1)}{4k} \int_{\Omega} \frac{ |\nabla v|^{2p-4}}{v^{k}} \nabla v \cdot \nabla |\nabla v|^2.
\end{align}
By \eqref{eq-4-41} again,
\begin{equation*}
   - \int_{\Omega} \frac{|\nabla v|^{2p-4}}{v^{k}} \nabla v \cdot \nabla |\nabla v|^2 \leq  \frac{\mu }{p-1} \int_{\Omega}\frac{|\nabla v|^{2p-2}}{v^{k-1}}.
\end{equation*}
This together with \eqref{new-eq8} implies that
\begin{equation*}
    \int_{ \Omega}  \left|\nabla \left( \frac{|\nabla v|^{p-1}}{v^{\frac{k-1}{2}}}\right)\right|^2 \leq \frac{(p-1)^2}{4} \int_{ \Omega} \frac{|\nabla v|^{2p-6}}{v^{k-1}}\big |\nabla |\nabla v|^2\big|^2 + \frac{\mu(k-1)}{2} \int_{\Omega}\frac{|\nabla v|^{2p-2}}{v^{k-1}}.
\end{equation*}
This together with \eqref{new-eq6} implies that
\begin{equation*}
     \int_{\partial \Omega} \frac{|\nabla v|^{2p-4}}{v^{k-1}} \frac{\partial |\nabla v|^2}{\partial \nu} \leq \epsilon \int_{ \Omega} \frac{|\nabla v|^{2p-6}}{v^{k-1}}\big |\nabla |\nabla v|^2\big|^2 + C_1 \int_{ \Omega} \frac{|\nabla v|^{2p-2}}{v^{k-1}},
\end{equation*}
where $C_1=\frac{2\mu(k-1)}{(p-1)^2} + C(p,\epsilon,K_1).$ The lemma is thus proved.
\end{proof}

\begin{lemma}
\label{lem-4-3}
Let $p \ge  3$ and $k \ge p-1.$  For every given $\varepsilon_1>0$ and $\varepsilon_2>0$, there is $M_1=M_1(\varepsilon_1,\varepsilon_2,p,k,\nu)$ such that
\begin{equation}
    \label{eq-4-30}
    \int_{\Omega}\frac{|\nabla v|^{2p-4}}{v^{k-1}} \nabla u \cdot \nabla v \leq M_1\int_{\Omega} \frac{u^p}{v^{k-p+1}} + {\varepsilon}_1 \int_{\Omega}\frac{|\nabla v|^{2p-6}}{v^{k-1}}\big |\nabla |\nabla v|^2\big|^2 + {\varepsilon_2} \int_{\Omega}\frac{|\nabla v|^{2p}}{v^{k+1}}
    \end{equation}
for all $t \in (0, T_{\rm max})$.
\end{lemma}

\begin{proof}
 First,  we have that
\begin{align*}
    \int_{\Omega}\frac{|\nabla v|^{2p-4}}{v^{k-1}} \nabla u \cdot \nabla v
  = \underbrace{(k-1) \int_{\Omega}\frac{u|\nabla v|^{2p-2}}{v^{k}}}_{J_{1}} - \underbrace{\int_{\Omega}\frac{u|\nabla v|^{2p-4}}{v^{k-1}}\Delta v}_{J_{2}}   - \underbrace{ (p-2) \int_{\Omega} \frac{u |\nabla v|^{2p-6}}{v^{k-1}} \nabla v \cdot \nabla |\nabla v|^2}_{J_{3}}.
\end{align*}

Next, by Young's inequality, for every $B_1>0$, there exists a positive constant $A_1=A_1(k,p,B_1)>0$ such that
\begin{align}
\label{eq-4-31}
     J_1=(k-1) \int_{\Omega}\frac{u|\nabla v|^{2p-2}}{v^{k}}&=(k-1) \int_{\Omega} \frac{u}{v^{\frac{k-p+1}{p}}} \cdot  \frac{|\nabla v|^{2p-2}}{v^{\frac{(k+1)(p-1)}{p}}}\nonumber\\
     & \leq A_1 \int_{\Omega} \frac{u^p}{v^{k-p+1}}+ B_1 \int_{\Omega} \frac{|\nabla v|^{2p}}{v^{k+1}}.
\end{align}

Now, by the fact that $\Delta v =\mu v-\nu u$, and Young's inequality, for every $B_2>0$,
 there exists a positive constant $A_2=A_2(k,p,\nu,B_2)>0$ such that
\begin{align}
    \label{eq-4-32}
    {-J_2}=-\int_{\Omega}\frac{u|\nabla v|^{2p-4}}{v^{k-1}}\Delta v &=-\mu \int_{\Omega}\frac{u|\nabla v|^{2p-4}}{v^{k-2}} + \nu \int_{\Omega}\frac{u^2|\nabla v|^{2p-4}}{v^{k-1}} \nonumber\\
    & \leq \nu \int_{\Omega} \frac{u^2}{v^{\frac{2(k-p+1)}{p}}} \cdot \frac{|\nabla v|^{2p-4}}{v^{\frac{(p-2)(k+1)}{p}}} \nonumber\\
    & \leq A_2 \int_{\Omega} \frac{u^p}{v^{k-p+1}}+ B_2 \int_{\Omega} \frac{|\nabla v|^{2p}}{v^{k+1}}.
\end{align}

Finally, for every given   ${\varepsilon}_1>0$ and $B_3>0$,  there exist positive constants $A_3=A_3({\varepsilon}_1, k,p,\nu,B_3)>0$ and
$A_4=A_4(A_3,B_3)$  such that
\begin{align}
\label{eq-4-33}
    {-J_3}&= -(p-2) \int_{\Omega} \frac{u |\nabla v|^{2p-6}}{v^{k-1}} \nabla v \cdot \nabla |\nabla v|^2\nonumber\\
     & \leq  A_3 \int_{\Omega}\frac{u^2|\nabla v|^{2p-4}}{v^{k-1}} + {\varepsilon}_1 \int_{\Omega}\frac{|\nabla v|^{2p-6}}{v^{k-1}}\big |\nabla |\nabla v|^2\big|^2 \nonumber\\
    & \leq A_4 \int_{\Omega} \frac{u^p}{v^{k-p+1}}+ B_3 \int_{\Omega} \frac{|\nabla v|^{2p}}{v^{k+1}}+ {\varepsilon}_1 \int_{\Omega}\frac{|\nabla v|^{2p-6}}{v^{k-1}}\big |\nabla |\nabla v|^2\big|^2.
\end{align}
Combining  \eqref{eq-4-31}, \eqref{eq-4-32} and \eqref{eq-4-33}  with $B_1=B_2=B_3=\frac{1}{3} \varepsilon_2$,  we obtain \eqref{eq-4-30} with   $M_1=M_1({\varepsilon}_1,
  {\varepsilon}_2,k,p,\nu):=
  A_1+A_2+A_4$.
\end{proof}

In the following lemma, we provide some estimate for $\int_{\Omega} \frac{|\nabla v|^{2p-6}}{v^{k-1}}\big |\nabla |\nabla v|^2\big|^2$.

\begin{lemma}
\label{lem-4-6}
Let $p \ge 3$ and $k \ge p-1.$  For every  given $\varepsilon_1>0$ and $\varepsilon_2>0$, there are  $M_2=M_2(\varepsilon_1,\varepsilon_2,p,k,\nu)>0$ and $M_3=M_3(p,k,\varepsilon_1)>0$ such that
\begin{align*}
    (1-\varepsilon_1)  \int_{\Omega} \frac{|\nabla v|^{2p-6}}{v^{k-1}}\big |\nabla |\nabla v|^2\big|^2 \leq & M_2 \int_{\Omega} \frac{u^p}{v^{k-p+1}}  +\Bigg(\frac{4(k-1)^2}{(2p-3)^2}+\varepsilon_2\Bigg) \int_{\Omega} \frac{|\nabla v|^{2p}}{v^{k+1}}\nonumber\\
   & +\frac{4M_3-8 \mu}{2p-3} \int_{\Omega} \frac{|\nabla v|^{2p-2}}{v^{k-1}}
\end{align*}
{for all $t \in (0, T_{\rm max})$.}
\end{lemma}

\begin{proof}
First of all, recall \eqref {eq-4-43}, that is,
\begin{align}
\label{eq-4-43-0}
  (p-2)  \int_{\Omega} \frac{|\nabla v|^{2p-6}}{v^{k-1}}\big |\nabla |\nabla v|^2\big|^2 &= \int_{\partial \Omega} \frac{|\nabla v|^{2p-4}}{v^{k-1}} \frac{\partial |\nabla v|^2}{\partial \nu} + (k-1)\int_{\Omega} \frac{ |\nabla v|^{2p-4}}{v^{k}} \nabla v \cdot \nabla |\nabla v|^2 \nonumber\\
    & \quad -  \int_{\Omega} \frac{|\nabla v|^{2p-4}}{v^{k-1}}  \Delta |\nabla v|^2.
\end{align}
By Lemma \ref{lem-4-4},  {for every $\varepsilon_1>0$, there exists} $M_3=M_3(p,k, \varepsilon_1)>0$ such that
\begin{equation}
\label{eq-4-44}
     \int_{\partial \Omega} \frac{|\nabla v|^{2p-4}}{v^{k-1}} \frac{\partial |\nabla v|^2}{\partial \nu} \leq \frac{(2p-3)\varepsilon_1}{8} \int_{ \Omega} \frac{|\nabla v|^{2p-6}}{v^{k-1}}\big |\nabla |\nabla v|^2\big|^2 + M_3 \int_{ \Omega} \frac{|\nabla v|^{2p-2}}{v^{k-1}}.
\end{equation}
By \eqref{eq-4-45},
\begin{align*}
    -  \int_{\Omega} \frac{|\nabla v|^{2p-4}}{v^{k-1}}  \Delta |\nabla v|^2
   =-2 \mu \int_{\Omega} \frac{|\nabla v|^{2p-2}}{v^{k-1}}+2 \nu \int_{\Omega} \frac{|\nabla v|^{2p-4}}{v^{k-1}} \nabla u \cdot \nabla v -2  \int_{\Omega} \frac{|\nabla v|^{2p-4}}{v^{k-1}} |D^2v|^2.
\end{align*}
Note  that
\begin{equation*}
  \big |\nabla |\nabla v|^2\big|^2 = 4\sum_{i=1}^n |\nabla v\cdot\nabla v_{x_i} |^2 \leq 4 |\nabla v|^2 \sum_{i=1}^n {|\nabla v_{x_i}|^2}= 4|\nabla v|^2 |D^2v|^2.
\end{equation*}
Hence
\begin{equation*}
    -2\int_{\Omega} \frac{|\nabla v|^{2p-4}}{v^{k-1}} |D^2v|^2 \leq -\frac{1}{2} \int_{\Omega} \frac{|\nabla v|^{2p-6}}{v^{k-1}}\big |\nabla |\nabla v|^2\big|^2.
\end{equation*}
It then follows that
\begin{align}
    \label{eq-4-47}
    -  \int_{\Omega} \frac{|\nabla v|^{2p-4}}{v^{k-1}}  \Delta |\nabla v|^2  \leq& -2 \mu \int_{\Omega} \frac{|\nabla v|^{2p-2}}{v^{k-1}}+2 \nu \int_{\Omega} \frac{|\nabla v|^{2p-4}}{v^{k-1}} \nabla u \cdot \nabla v\nonumber\\
     & - \frac{1}{2}  \int_{\Omega} \frac{|\nabla v|^{2p-6}}{v^{k-1}}\big |\nabla |\nabla v|^2\big|^2.
\end{align}

Next, by Young's inequality, we have that
\begin{equation}
\label{eq-4-48}
    (k-1)\int_{\Omega} \frac{ |\nabla v|^{2p-4}}{v^{k}} \nabla v \cdot \nabla |\nabla v|^2 \leq \frac{(2p-3)}{4}\int_{\Omega}\frac{|\nabla v|^{2p-6}}{v^{k-1}}\big |\nabla |\nabla v|^2\big|^2 + \frac{(k-1)^2}{2p-3} \int_{\Omega}\frac{|\nabla v|^{2p}}{v^{k+1}}.
\end{equation}
Substituting \eqref{eq-4-44}, \eqref{eq-4-47} and \eqref{eq-4-48} into \eqref{eq-4-43-0} yields that
\begin{align*}
    (p-2)\int_{\Omega} \frac{|\nabla v|^{2p-6}}{v^{k-1}}\big |\nabla |\nabla v|^2\big|^2 & \leq \frac{(2p-3)\varepsilon_1}{8} \int_{ \Omega} \frac{|\nabla v|^{2p-6}}{v^{k-1}}\big |\nabla |\nabla v|^2\big|^2 + M_3 \int_{ \Omega} \frac{|\nabla v|^{2p-2}}{v^{k-1}}\nonumber\\
     &\quad + \frac{(2p-3)}{4}\int_{\Omega}\frac{|\nabla v|^{2p-6}}{v^{k-1}}\big |\nabla |\nabla v|^2\big|^2 + \frac{(k-1)^2}{2p-3} \int_{\Omega}\frac{|\nabla v|^{2p}}{v^{k+1}} \nonumber\\
    &\quad -2 \mu \int_{\Omega} \frac{|\nabla v|^{2p-2}}{v^{k-1}} +2 \nu \int_{\Omega} \frac{|\nabla v|^{2p-4}}{v^{k-1}} \nabla u \cdot \nabla v \nonumber\\
   &\quad  - \frac{1}{2}  \int_{\Omega} \frac{|\nabla v|^{2p-6}}{v^{k-1}}\big |\nabla |\nabla v|^2\big|^2,
\end{align*}
which implies
\begin{align}
    \label{eq-4-50}
     \left(1-\frac{\varepsilon_1}{2}\right) \int_{\Omega} \frac{|\nabla v|^{2p-6}}{v^{k-1}}\big |\nabla |\nabla v|^2\big|^2  \leq & \frac{4(k-1)^2}{(2p-3)^2} \int_{\Omega}\frac{|\nabla v|^{2p}}{v^{k+1}} + \frac{4M_3-8 \mu}{2p-3} \int_{\Omega} \frac{|\nabla v|^{2p-2}}{v^{k-1}}\nonumber\\
    & + \frac{8 \nu}{2p-3} \int_{\Omega} \frac{|\nabla v|^{2p-4}}{v^{k-1}} \nabla u \cdot \nabla v.
\end{align}
By Lemma \ref{lem-4-3}, {for every given  $\varepsilon_1>0$ and $\varepsilon_2>0$}, there is $M_2=M_2(\varepsilon_1,\varepsilon_2, p,k,\nu)$ such that
\begin{align*}
    \int_{\Omega}\frac{|\nabla v|^{2p-4}}{v^{k-1}} \nabla u \cdot \nabla v & \leq \frac{(2p-3)M_2}{8\nu}\int_{\Omega} \frac{u^p}{v^{k-p+1}} + \frac{(2p-3)\varepsilon_1}{16 \nu} \int_{\Omega}\frac{|\nabla v|^{2p-6}}{v^{k-1}}\big |\nabla |\nabla v|^2\big|^2\nonumber\\
    &\quad+ \frac{(2p-3)\varepsilon_2}{8 \nu} \int_{\Omega}\frac{|\nabla v|^{2p}}{v^{k+1}}.
\end{align*}
This together with \eqref{eq-4-50} completes the proof.
\end{proof}

Now we prove Proposition \ref{main-prop2}.

\begin{proof}[Proof of Proposition \ref{main-prop2}]
 By Lemma \ref{lem-4-5} and Lemma \ref{lem-4-6}, for every given  $\varepsilon_1>0$ and $\varepsilon_2>0$, there are $M_2, M_3>0$ such that
 \begin{align*}
    \int_{\Omega}\frac{|\nabla v|^{2p}}{v^{k+1}} &\leq \frac{(p-1)^2}{k^2} \int_{\Omega}\frac{|\nabla v|^{2p-6}}{v^{k-1}}\big |\nabla |\nabla v|^2\big|^2 + \frac{2 \mu}{k} \int_{\Omega}\frac{|\nabla v|^{2p-2}}{v^{k-1}}\\
    & \leq \frac{(p-1)^2}{k^2} \Bigg[\frac{\frac{4(k-1)^2}{(2p-3)^2}+\varepsilon_2}{1-\varepsilon_1} \int_{\Omega} \frac{|\nabla v|^{2p}}{v^{k+1}} + \frac{M_2}{1-\varepsilon_1} \int_{\Omega} \frac{u^p}{v^{k-p+1}}\\
    &\quad+ \frac{4M_3-8\mu}{(1-\varepsilon_1)(2p-3)} \int_{\Omega} \frac{|\nabla v|^{2p-2}}{v^{k-1}} \Bigg]+ \frac{2 \mu}{k} \int_{\Omega}\frac{|\nabla v|^{2p-2}}{v^{k-1}}\\
    & = \frac{(p-1)^2}{k^2} \Bigg[\frac{\frac{4(k-1)^2}{(2p-3)^2}+\varepsilon_2}{1-\varepsilon_1} \int_{\Omega} \frac{|\nabla v|^{2p}}{v^{k+1}} + \frac{M_2}{1-\varepsilon_1} \int_{\Omega} \frac{u^p}{v^{k-p+1}}\Bigg]\\
    &\quad+  \Bigg[\frac{2 \mu}{k} + \frac{(4M_3-8\mu)(p-1)^2}{(1-\varepsilon_1)(2p-3)k^2} \Bigg] \int_{\Omega}\frac{|\nabla v|^{2p-2}}{v^{k-1}}
\end{align*}
for all $t \in (0, T_{\rm max})$. Since $p-1 \leq k <  2p-2$, there exists positive constants $\varepsilon_1>0$ and $\varepsilon_2>0$ such that
\begin{equation*}
    \frac{(p-1)^2}{k^2} \cdot \frac{1}{1-\varepsilon_1} \cdot  \Bigg(\frac{4(k-1)^2}{(2p-3)^2}+\varepsilon_2 \Bigg) <1.
\end{equation*}
Therefore, we have for some $ M=M(p,k,\nu,\varepsilon,M_2)>0$ and $M_4>0$
$$\int_{\Omega} \frac{|\nabla v|^{2p}}{v^{k+1}} \leq \frac{M}{2} \int_{\Omega} \frac{u^p}{v^{k-p+1}}+M_4 \int_{\Omega}\frac{|\nabla v|^{2p-2}}{v^{k-1}}.$$
An application Young's inequality on the latter integral above entails that
\begin{align*}
   \int_{\Omega}\frac{|\nabla v|^{2p-2}}{v^{k-1}} &= \int_{\Omega} \frac{|\nabla v|^{2p-2}}{v^{\frac{(k+1)(p-1)}{p}}} \cdot  \frac{1}{v^{\frac{k-2p+1}{p}}}\\
   & \leq \frac{1}{2M_4} \int_{\Omega}\frac{|\nabla v|^{2p}}{v^{k+1}} + \frac{M^{**}}{2} \int_{\Omega} v^{2p-k-1},
\end{align*}
where  $M^{**}=M^{**}(p,k,\mu,\nu,\varepsilon_1,\varepsilon_2,M_2,M_3,M_4)>0$. Thus we have
$$\int_{\Omega} \frac{|\nabla v|^{2p}}{v^{k+1}} \leq M^* \int_{\Omega} \frac{u^p}{v^{k-p+1}}+M^{**}\int_{\Omega} v^{2p-k-1},$$
where
$$
M^*=M^*(p,k,\nu,\varepsilon_1,\varepsilon_2,M_2):=\frac{M_2}{1-\varepsilon_1}\frac{(p-1)^2}{k^2}\frac{1}{1- \frac{(p-1)^2}{k^2} \cdot \frac{1}{1-\varepsilon_1} \cdot  \Big(\frac{4(k-1)^2}{(2p-3)^2}+\varepsilon_2 \Big)}.
$$
The proposition thus follows.
\end{proof}

\subsection{Proof of Theorem \ref{main-thm1}}

In this subsection, we prove Theorem \ref{main-thm1}.

\begin{proof} [Proof of Theorem \ref{main-thm1}]
(1) First of all, { for the simplicity, we put $T_{\max}=T_{\max}^{\sigma}(u_0)$ and  $(u,v)=(u_{\sigma}(x,t;u_0),v_{\sigma}(x,t;u_0))$ for $x\in\Omega$ and $t\in [0,T_{\max})$. We let $p>\max\{2,n\}.$}

Recall that, multiplying the first equation in \eqref{main-eq} by $u^{p-1}$ and integrating  over $\Omega$ with respect to $x$, we have
\begin{align}
\label{u-diff-t-eq}
    \frac{1}{p} \cdot \frac{d}{dt} \int_{\Omega} u^p &=-(p-1)\int_{\Omega} u^{p-2}|\nabla u|^2 + (p-1) \chi \int_{\Omega} \frac{u^{p-1}}{v}\nabla u \cdot \nabla v \nonumber\\
    &\quad+\int_{\Omega} a(\cdot,t) u^p -  \int_{\Omega} b(\cdot,t) { u^{p+1+\sigma}} \quad \text{for all} \;\; t\in (0,T_{\max}).
\end{align}
Note that, by Proposition \ref{mass-persistence-prop}(1), for every $T>0$, there is $\delta=\delta(T)>0$ such that
\begin{equation}
    \label{v-delta*eq-1}
    v(x,t)\ge \delta \quad \forall\,\, x\in\Omega,\,\, t\in (0,\min\{T,T_{\max}\}).
\end{equation}
By Proposition \ref{main-prop2}, there are $C^*=C^*(p, \nu)>0$ and $C^{**}=C^{**}(p,\mu,\nu)$ such that \begin{equation}
\label{C-star-eq}
   \int_{\Omega} \frac{|\nabla v|^{2p+2}}{v^{p+1}} \leq C^* \int_{\Omega} u^{p+1}+C^{**}\int_{\Omega} v^{p+1}\quad \text{for all}\;\, t \in (0, T_{\max}).
\end{equation}
{By Lemma \ref{prelim-lm8}, for any $\varepsilon>0$, there is $C=C(\varepsilon,p)>0$ such that
\begin{equation}
\label{u-v-est-lp}
\int_\Omega v^{p+1}\le \varepsilon \int_\Omega u^{p+1}+ C(\varepsilon,p)\Big(\int_\Omega u\Big)^{p+1}\quad \text{for all} \;\; t\in (0,T_{\max}).
\end{equation}
By Young's inequality, for any $\varepsilon>0$, there is $\tilde C(\varepsilon,p)>0$ such that
\begin{equation}
\label{u-v-est-lp-1}
\Big(\int_\Omega u^{p+1}\Big)^{\frac{p}{p+1}}\int_\Omega u\le  \varepsilon \int_\Omega u^{p+1}+\tilde C(\varepsilon,p)\Big(\int_\Omega u\Big)^{p+1}.
\end{equation}
Therefore, by \eqref{aux-new-eq00}, \eqref{v-delta*eq-1}, \eqref{C-star-eq}, \eqref{u-v-est-lp}, and
 \eqref{u-v-est-lp-1}, we arrive at
\begin{align*}
   &\chi \int_{\Omega} \frac{u^{p-1}}{v}\nabla u \cdot \nabla v\\
&\le\int_\Omega u^{p-2}|\nabla u|^2+\frac{\chi^2}{4\delta} \Big(\int_\Omega u^{p+1}\Big)^{\frac{p}{p+1}}\Big(\int_\Omega \frac{|\nabla v|^{2p+2}}{v^{p+1}}\Big)^{\frac{1}{p+1}} \quad \qquad\qquad\qquad\quad  { \text{(by \eqref{aux-new-eq00}, \eqref{v-delta*eq-1})}}\nonumber  \\
&\le  \int_\Omega u^{p-2}|\nabla u|^2+\frac{\chi^2}{4\delta} \Big(\int_\Omega u^{p+1}\Big)^{\frac{p}{p+1}} \Big(C^*\int_\Omega u^{p+1}+C^{**}\int_\Omega v^{p+1}\Big)^{\frac{1}{p+1}}\quad\quad \,\,{ \text{(by \eqref{C-star-eq}})}\nonumber\\
&\le \int_\Omega u^{p-2}|\nabla u|^2+\frac{\chi^2 (C^*)^{\frac{1}{p+1}}}{4\delta} \int_\Omega u^{p+1} +\frac{\chi^2 (C^{**})^{\frac{1}{p+1}}}{4\delta}  \Big(\int_\Omega u^{p+1}\Big)^{\frac{p}{p+1}}\Big(\int_\Omega v^{p+1}\Big)^{\frac{1}{p+1}} \nonumber\\
&{ \le \int_\Omega u^{p-2}|\nabla u|^2+\frac{\chi^2 (C^*)^{\frac{1}{p+1}}}{4\delta} \int_\Omega u^{p+1}} \nonumber\\
&{\quad+\frac{\chi^2 (C^{**})^{\frac{1}{p+1}}}{4\delta}  \Big(\int_\Omega u^{p+1}\Big)^{\frac{p}{p+1}}\Big[ \varepsilon \int_\Omega u^{p+1}+ C(\varepsilon,p)\Big(\int_\Omega u\Big)^{p+1}\Big]^{\frac{1}{p+1}}  \quad\qquad\,\,\, \,   \text{by \eqref{u-v-est-lp}}}\nonumber\\
&{ \le \int_\Omega u^{p-2}|\nabla u|^2+\frac{\chi^2 (C^*)^{\frac{1}{p+1}}}{4\delta} \int_\Omega u^{p+1}} \nonumber\\
&{ \quad+\frac{\chi^2 (C^{**})^{\frac{1}{p+1}}}{4\delta}  \Big(\int_\Omega u^{p+1}\Big)^{\frac{p}{p+1}}\Big[( \varepsilon)^{\frac{1}{p+1}}\big(\int_\Omega u^{p+1}\big)^{\frac{1}{p+1}}+ \big(C(\varepsilon,p)\big)^{\frac{1}{p+1}}\int_\Omega u\Big]} \quad\qquad\,\,\,\\
&\le\int_\Omega u^{p-2}|\nabla u|^2+\frac{\chi^2}{4\delta}\Big\{ (C^*)^{\frac{1}{p+1}}+ \big(C^{**} \varepsilon\big)^{\frac{1}{p+1}}+\big(C^{**}C(\varepsilon, p)\big)^{\frac{1}{p+1}}\varepsilon\Big\} \int_\Omega u^{p+1} \nonumber\\
&\quad + \frac{\chi^2}{4\delta} \big(C^{**} C(\varepsilon, p)\big)^{\frac{1}{p+1}}\tilde C(\varepsilon,p)
\Big(\int_\Omega u\Big)^{{p+1}} \quad
\forall\, t\in (0,\min\{T,T_{\max}\})\quad\qquad  { \text{by \eqref{u-v-est-lp-1}}}
\end{align*}
This together with \eqref{u-diff-t-eq}, we arrive at
\begin{align}
\label{u-diff-t-eq-2}
\frac{1}{p} \cdot \frac{d}{dt} \int_{\Omega} u^p
    &\leq \frac{(p-1)\chi^2}{4\delta}\Big\{ (C^*)^{\frac{1}{p+1}}+ \big(C^{**} \varepsilon\big)^{\frac{1}{p+1}}+\big(C^{**} C(\varepsilon,p)\big)^{\frac{1}{p+1}}\varepsilon\Big\} \int_\Omega u^{p+1} \nonumber\\
    &\quad + \frac{(p-1)\chi^2 }{4\delta}
    \big(C^{**} C(\varepsilon, p)\big)^{\frac{1}{p+1}}\tilde  C(\varepsilon,p)\Big(\int_\Omega u\Big)^{{ p+1}}\nonumber\\
    &\quad  +a_{\sup}\int_{\Omega}u^p -  b_{\inf} \int_{\Omega} u^{p+1+\sigma}
 \quad \text{for all}\;\, t\in (0,\min\{T,T_{\max}\}).
\end{align}
Observe that, by Young's inequality, there is $C_1^*= C_1^*(p,b_{\inf},\varepsilon,\sigma,\mu,\nu,\delta,|\Omega|)>0$ such that
\begin{equation}
\label{eq-4-52}
    \frac{(p-1)\chi^2}{4\delta} \Big\{ (C^**)^{\frac{1}{p+1}}+ \big(C^{**} \varepsilon\big)^{\frac{1}{p+1}}+\big(C^{**} C(\varepsilon,p)\big)^{\frac{1}{p+1}}\varepsilon\Big\}\int_\Omega u^{p+1}\le \frac{b_{\inf}}{2}\int_\Omega u^{p+1+\sigma}+C_1^*,
\end{equation}
and by H\"older's inequality, there is $C_2^*=C_2^*(p,a_{\sup},b_{\inf},\varepsilon,\sigma,\mu,\nu,|\Omega|)>0$ such that
\begin{equation}
\label{eq-4-53}
\Big(\frac{1}{p}+a_{\sup}\Big)\int_{\Omega}u^p \leq \frac{b_{\inf}}{2}\int_\Omega u^{p+1+\sigma}+C_2^*,
\end{equation}
and by  Lemma \ref{prelim-lm4},
\begin{equation}
\label{eq-4-54}
    \frac{(p-1)\chi^2 }{4\delta}
    \big(C^{**} C(\varepsilon, p)\big)^{\frac{1}{{ p+1}}} \tilde C(\varepsilon,p)\Big(\int_\Omega u\Big)^{{ p+1}} \leq \frac{(p-1)\chi^2 }{4\delta}
    \big(C^{**}C\big)^{\frac{1}{p+1}} \tilde C(m_0^*)^{{p+1}},
\end{equation}
for all $t\in (0,\min\{T,T_{\max}\})$. Hence, with the help of \eqref{u-diff-t-eq-2}, \eqref{eq-4-52}, \eqref{eq-4-53} and \eqref{eq-4-54}, we get that
\begin{equation*}
   \frac{d}{dt} \int_{\Omega} u^p \leq  - \int_{\Omega} u^p + C \quad \text{for all}\,\,  t\in \big(0,\min\{T,T_{\max}\}\big)
\end{equation*}
for some $C>0$. Let us {denote} $y(t):= \int_{\Omega} u^p(t,x)dx$. Then we obtain
\begin{equation*}
    y'(t) \leq - y(t)+  C \quad \text{for all}\,\,  t\in \big(0,\min\{T,T_{\max}\}\big).
\end{equation*}
Thus, the comparison principle for scalar ODEs entails the boundedness of $y(t)$ on $(0, \min\{T,T_{\max}\}),$ that is,
\begin{equation*}
    \sup_{t\in [0,\min\{T,T_{\max}\})}\int_\Omega u^p<\infty.
\end{equation*}
This proves (1).

\medskip

(2)  Put $T_{\max}=T_{\max}^{\sigma}(u_0)$ and  $(u,v)=(u_{\sigma}(x,t;u_0),v_{\sigma}(x,t;u_0))$ for $x\in\Omega$ and $t\in [0,T_{\max})$.
By Proposition \ref{mass-persistence-prop}(2), there is $\delta>0$ such that
\begin{equation*}
    v(x,t)\ge \delta \quad \forall\,\, x\in\Omega,\,\, t\in (0,T_{\max}).
\end{equation*}
It  then follows from the arguments in (1)  that
\begin{equation*}
    \sup_{t\in [0,T_{\max})}\int_\Omega u^p<\infty.
\end{equation*}
This proves (2).

\medskip

(3) First, put $(u(x,t),v(x,t))=(u(x,t;u_0),v(x,t;u_0))$.  Let
$\beta^*=\chi+2-2\sqrt{\chi+1}$. Then
\begin{equation*}
    p^*:=\frac{4\beta^*}{(\chi-\beta^*)^2}=\frac{4(\chi+2-2\sqrt{\chi+1})}{(-2+2\sqrt{\chi+1})^2}=1,
\end{equation*}
and
\begin{equation*}
    \frac{(p^*+1)\beta^*\mu}{p^*}=2\beta^*\mu=2(\chi+2-2\sqrt{\chi+1})\mu=a_{\chi,\mu}<a_{\inf}.
\end{equation*}
 By the arguments of Lemma \ref{mass-lem-3-4}, for any $u_0\ge 0$ satisfying \eqref{u-0-assumption-eq},
we have
\begin{equation*}
    \frac{d }{dt}\int_\Omega  u^{-1}(x,t)dx\le -(a_{\inf}-a_{\chi,\mu})\int_\Omega u^{-1}(x,t)dx+b_{\sup}|\Omega| \quad \text{for all}\,\,   t\in (0,T_{\max}),
\end{equation*}
and then
\begin{equation*}
    \int_\Omega u^{-1} (x,t)dt\le \max\Big\{\int_\Omega u^{-1}(x,\tau_0)dx, \frac{b_{\sup}|\Omega|}{a_{\inf}-a_{\chi,\mu}}\Big\}= \frac{b_{\sup}|\Omega|}{(a_{\inf}-a_{\chi,\mu})C_{\chi}} \quad \text{for all}\,\,  \tau_0<t<T_{\max}.
\end{equation*}
This together with Lemma \ref{prelim-lm1} implies that
\begin{equation*}
    \int_\Omega u\ge \frac{(a_{\inf}-a_{\chi,\mu})C_\chi}{b_{\sup}}|\Omega|,\quad
v\ge \frac{\delta^*\nu  (a_{\inf}-a_{\chi,\mu})C_{\chi}}{b_{\sup}}|\Omega|
\quad \text{for all}\,\,   t\in [\tau_0,T_{\max}).
\end{equation*}
By the arguments of \eqref{u-diff-t-eq-2}, for any $p>\max\{2,n\}$, we  have
\begin{align*}
    \frac{1}{p} \cdot \frac{d}{dt} \int_{\Omega} u^p
    &\leq \frac{b_{\sup}|\Omega| (p-1)\chi^2\Big\{(C^*)^{\frac{1}{p+1}}+ \big(C^{**} \varepsilon\big)^{\frac{1}{p+1}}+\big(C^{**} C(\varepsilon,p)\big)^{\frac{1}{p+1}}\varepsilon\Big\}}{4\delta^*\nu  (a_{\inf}-a_{\chi,\mu})C_{\chi}} \int_\Omega u^{p+1} \nonumber\\
    &\quad + \frac{b_{\sup}|\Omega|(p-1)\chi^2 \big(C^{**} C(\varepsilon, p)\big)^{\frac{1}{p+1}} \tilde C(\varepsilon,p)}{4\delta^*\nu (a_{\inf}-a_{\chi,\mu})C_{\chi}}
    \Big(\int_\Omega u\Big)^{{p+1}}\nonumber\\
    &\quad  +a_{\sup}\int_{\Omega}u^p -  b_{\inf} \int_{\Omega} u^{p+1}
    \quad \text{for all}\,\,  t\in (\tau,T_{\max}).
\end{align*}
By the assumption \eqref{a-assumption-eq1},  there  are $p>\max\{2,n\}$ and $0<\varepsilon\ll 1$ such that
\begin{equation*}
    \frac{b_{\sup}|\Omega| (p-1)\chi^2\Big\{ (C^*(n,\nu))^{\frac{1}{p+1}}+ \big(C^{**}(n,\mu,\nu) \varepsilon\big)^{\frac{1}{p+1}}+\big(C^{**}(n,\mu,\nu)\tilde C(\varepsilon,p)\big)^{\frac{1}{p+1}}\varepsilon\Big\}}{4\delta^*\nu (a_{\inf}-a_{\chi,\mu})C_{\chi}}  <b_{\inf}.
\end{equation*}
Note that
\begin{equation*}
    \int_\Omega u^{p+1}\ge \frac{1}{|\Omega|}\Big(\int_\Omega u^p\Big)^{\frac{p+1}{p}}\quad \text{for all}\,\,  t\in (0,T_{\max}).
\end{equation*}
Therefore, there is $b_1>0$ such that
\begin{align*}
    \frac{1}{p} \cdot \frac{d}{dt} \int_{\Omega} u^p
    &\leq  \frac{b_{\sup}|\Omega|(p-1)\chi^2 \big(C^{**}(n,\mu,\nu) C(\varepsilon, p)\big)^{\frac{1}{p+1}} \tilde C(\varepsilon,p)}{4\delta^*\nu (a_{\inf}-a_{\chi,\mu})C_{\chi}}
\Big(\int_\Omega u\Big)^{{ p+1}}\nonumber\\
&\quad  +a_{\sup}\int_{\Omega}u^p -  b_1\Big( \int_{\Omega} u^{p}\Big)^{\frac{p+1}{p}}
\quad \text{for all}\,\,  t\in (\tau,T_{\max}).
\end{align*}
This together with comparison principle for scalar ODE's, Lemma \ref{prelim-lm4}, and the continuity of $\int_\Omega u^p(x,t;u_0)dx$ in $t\in [0,T_{\max})$ implies that
\begin{equation*}
    \sup_{0\le t<T_{\max}(u_0)}\|u(\cdot,t;u_0)\|_{L^p}<\infty.
\end{equation*}
This proves (3). }
\end{proof}

\section{Global existence and boundedness}

In this section, we study the global existence and prove Theorem \ref{main-thm2}. Throughout this section, we put
{ $u(x,t)=u_\sigma(x,t;u_0)$, $v(x,t)=v_\sigma(x,t;u_0)$, and $T_{\max}=T^{\sigma}_{\max}(u_0)$,}  and assume that $u_0$ satisfies \eqref{initial-cond-eq}.

\begin{proof}[Proof of Theorem \ref{main-thm2}]
(1) {For simplicity, let $T_{\max}=T_{\max}^{\sigma}(u_0)$ and  $(u,v)=(u_{\sigma}(x,t;u_0),v_{\sigma}(x,t;u_0))$ for $x\in\Omega$ and $t\in [0,T_{\max})$.} We prove Theorem \ref{main-thm2} (1)  by contradiction. Assume that $T_{\max}<\infty$.  Then by Proposition \ref{mass-persistence-prop}(1), there is $\delta>0$ such that
\begin{equation}
\label{eq-5-1}
    v(x,t) \ge \delta \quad \quad \text{\rm for all}\,\, x\in \Omega \quad \text{\rm and}\,\, t\in (0,T_{\max}).
\end{equation}
By  Proposition \ref{main-prop}, we then have
\begin{equation}
\label{eq-5-2}
\limsup_{t \nearrow T_{\max}} \left\| u(\cdot,t) \right\|_{C^0(\bar \Omega)}  =\infty.
\end{equation}
By Theorem \ref{main-thm1},  there is $\bar p>n$ such that
\begin{equation*}
    \sup_{0\le t< T_{\max}}\|u(\cdot,t)\|_{L^{\bar p}}<\infty.
\end{equation*}
This implies that
\begin{equation*}
    \sup_{0\le t< T_{\max}}\|u(\cdot,t)\|_{L^{ p}}<\infty\quad \forall\,\, 1\le p\le n.
\end{equation*}

Fix a $p$ such that $\frac{n}{2}<p<n$. Then one can find $q>n$  satisfying
\begin{equation*}
    \frac{1}{p}-\frac{1}{n} < \frac{1}{q},
\end{equation*}
which allows us to find a positive constant $ { h} \in (1,\infty)$ such that
\begin{equation*}
    \frac{1}{{ h}}<1-\frac{(n-p)q}{np}.
\end{equation*}

By the variation of constant formula and the comparison principle for parabolic equations, we have
\begin{align*}
    u(\cdot,t) & \leq \underbrace{ e^{-t A}u_0}_{u_1(\cdot,t)}- \underbrace{\chi \int_{0}^{t} e^{-(t-s)A} \nabla \cdot \Big(\frac{u(\cdot,s)}{v(\cdot,s)}\nabla v(\cdot,s) \Big)ds }_{u_2(\cdot,t)}\\
    &\quad + \underbrace{\int_{0}^{t} e^{-(t-s)A} u(\cdot,s){ \big(1+a_{\sup}- b_{\inf} u^{1+\sigma}(\cdot,s)\big)} ds}_{u_3(\cdot,t)} \quad \text{\rm for all}\,\, t\in (0,T_{\rm max}),
\end{align*}
where $A=-\Delta +I: \mathcal{D}(A)\subset L^p(\Omega)\to L^p(\Omega)$ with $\mathcal{D}(A)=\{u\in W^{2,p}(\Omega)\,|\, \frac{\p u}{\p n}=0$ on $\p \Omega\}$, and
$e^{-tA}$ is the analytic semigroup generated by $-A$ on $L^p(\Omega)$ (see section 2).

In the following, we give estimates for $u_i(x,t)$ ($i=1,2,3$). Note that
\begin{equation}
    \label{eq-5-3}
    \|u_1(\cdot,t)\|_{ L^{\infty}(\Omega)}   \leq  \|u_0\|_{L^{\infty}(\Omega)}\quad \quad \text{for all}\,\,  t\in [0,T_{\rm max}).
\end{equation}
Note also that there exist $c_0,c_1>0$ such that {
\begin{equation*}
    u(x,s)\big(1+a_{\sup}-b_{\inf} u^{1+\sigma}(x,s)\big) \leq c_0-c_1 u^2(x,s)
\end{equation*}}
for all $s \in (0,T_{\max})$, $x \in \bar \Omega$.  Therefore, we have that
\begin{equation}
    \label{eq-5-4}
    u_3(\cdot,t) \leq  c_0 \int_{0}^{t} e^{-(t-s)A}  ds \leq c_0 \int_{0}^{t} e^{-(t-s)}  ds \leq c_0 \quad \text{\rm for all}\,\, t\in [0,T_{\rm max}).
\end{equation}

By  H\"older's inequality and \eqref{eq-5-1}, we have that
\begin{equation}
    \label{eq-5-5}
     \Big \| \frac{u(\cdot,s)}{v(\cdot,s)}  \nabla v(\cdot,s)\Big \|_{L^q(\Omega)}  \leq \frac{1}{\delta}\cdot \| u\|_{L^{{q h}}(\Omega)} \cdot \| \nabla v \|_{L^{{ \frac{q h}{h-1}}}(\Omega)}
\end{equation}
for all $t \in (0, T_{\rm max})$. Note that
\begin{equation*}
    { \frac{qh}{ h-1}=\frac{q}{1-\frac{1}{ h}}}<\frac{q}{(\frac{1}{p}-\frac{1}{n})q}=\frac{1}{\frac{1}{p}-\frac{1}{n}}=\frac{np}{n-p}.
\end{equation*}
Then by Lemma \ref{prelim-lm7} and Theorem \ref{main-thm1}, we have that
\begin{equation}
    \label{eq-5-6}
     \| \nabla v(\cdot,t) \|_{L^{{\frac{q h}{ h-1}}}(\Omega)} \leq C\| \nabla v(\cdot,t) \|_{L^{{ \frac{np}{n-p}}}(\Omega)} \leq C \|u(\cdot,t)\|_{L^p(\Omega)} \leq M_p\quad \forall\, t\in [0,T_{\max}).
\end{equation}
By H\"older's inequality and { Lemma \ref{prelim-lm4}}, we have
\begin{align}
    \label{eq-5-7}
    \| u\|_{L^{q h}(\Omega)} & \leq \| u\|_{L^{1}(\Omega)}^{{ 1-\lambda}} \cdot \| u\|_{L^{\infty}(\Omega)}^{ { \lambda}}\nonumber\\
    &\leq {(m^*_0)^{ 1-\lambda}} \cdot \| u\|_{L^{\infty}(\Omega)}^{ { \lambda}}
\end{align}
for all $t \in (0,T_{\max})$, where $ { \lambda=1-\frac{1}{q h}} \in (0,1)$.
By \eqref{eq-5-6} and \eqref{eq-5-7}, we have
\begin{equation}
    \label{eq-5-8}
     \Big \| \frac{u(\cdot,t)}{v(\cdot,t)}  \nabla v(\cdot,t)\Big \|_{L^q(\Omega)}  \leq {\tilde C^*} \| u(\cdot,t)\|_{L^{\infty}(\Omega)}^{{\lambda}}
\end{equation}
for all $t \in (0, T_{\rm max})$ with some constant ${\tilde C^*=\tilde C^*}(M_p,{ m^*_0},C, \lambda,\delta,q,h,p,n)>0.$ Choose $\beta \in \big(\frac{n}{2p},1\big)$ and fix any $T\in (0,T_{\max})$. There are $C_1,C_2,C_3,C_4>0$ and $\gamma>0$ such that
\begin{align}
\label{eq-5-9}
&\|u_2(\cdot,t)\|_{ L^{\infty}(\Omega)}\nonumber\\
  & \leq C_1 \|A^\beta u_2(\cdot,t)\|_{L^p} \qquad \qquad\qquad {\rm (by \,\, Lemma \,\, \ref{prelim-lm2})}\nonumber\\
 & \le C_1  \chi \int_{0}^t \Big \| A^{\beta}e^{-(t-s)A}\nabla\cdot \Big( \frac{u(\cdot,s)}{v(\cdot,s)}  \nabla v(\cdot,s)\Big)\Big \|_{L^p(\Omega)}ds \nonumber\\
 & \le C_1  \chi \int_{0}^t \big \| A^{\beta}e^{-\frac{t-s}{2}A}\big \|_{L^p(\Omega)}\cdot  \Big\|e^{-\frac{t-s}{2}A}\nabla\cdot \Big( \frac{u(\cdot,s)}{v(\cdot,s)}  \nabla v(\cdot,s)\big)\Big \|_{L^p(\Omega)}ds \nonumber\\
&\leq C_2 \chi \int_{0}^t (t-s)^{-\beta} \big(1+(t-s)^{-\frac{1}{2}}\big) e^{-\gamma (t-s)} \Big \| \frac{u(\cdot,s)}{v(\cdot,s)}  \nabla v(\cdot,s)\Big \|_{L^q(\Omega)}ds\qquad{\rm (by\,\, Lemmas\,\, \ref{prelim-lm2}\,\, and\,\,   \ref{prelim-lm3})} \nonumber\\
&\leq C_3 \chi \int_{0}^t (t-s)^{-\beta-\frac{1}{2}}  e^{-\gamma (t-s)} \| u(\cdot,s)\|_{L^{\infty}(\Omega)}^{{\lambda}}ds \nonumber\\
&\leq C_3 \chi \int_{0}^\infty  (t-s)^{-\beta-\frac{1}{2}}  e^{-\gamma (t-s)} \cdot \sup _{ s \in (0,T)}\| u(\cdot,s)\|_{L^{\infty}(\Omega)}^{{\lambda}}ds \nonumber\\
&\leq C_4 \chi  \sup _{ s \in (0,T)}\| u(\cdot,s)\|_{L^{\infty}(\Omega)}^{ {\lambda}}\quad \text{for all}\;\; t\in (0,T).
\end{align}
 Note that
\begin{equation*}
    u(x,t)>0\quad \forall\,\, x\in \Omega,\,\, t\in (0,T_{\max}).
\end{equation*}
We then have
\begin{equation*}
    \|u(\cdot,t)\|_{ L^{\infty}(\Omega)}   \leq  \Big(c_0+ \|u_0\|_{L^{\infty}(\Omega)}\Big) + K \cdot \Big(\sup _{ s \in (0,T)}\| u(\cdot,s)\|_{L^{\infty}(\Omega)}\Big)^{{\lambda}}
\end{equation*}
for all $t \in (0, T)$, where $K>0$ is a  positive constant and { $\lambda \in (0,1)$.} This
implies that
\begin{equation*}
    \sup_{0\le t\le T} \|u(t,\cdot)\|_{L^\infty}\le \Big(c_0+ \|u_0\|_{L^{\infty}(\Omega)}\Big) + K \cdot \Big(\sup _{ t \in (0,T)}\| u(\cdot,t)\|_{L^{\infty}(\Omega)}\Big)^{{\lambda}}
\end{equation*}
for all $T\in (0,T_{\max})$. Therefore,
\begin{equation*}
    \frac{\Big(c_0+ \|u_0\|_{L^{\infty}(\Omega)}\Big)}{\sup_{0\le t\le T} \|u(t,\cdot)\|_{{L^\infty}(\Omega)}}+\frac{K}{\sup_{0\le t\le T} \|u(t,\cdot)\|_{{L^\infty}(\Omega)}^{1- {\lambda}}}\ge 1.
\end{equation*}
This implies that
\begin{equation*}
    \limsup_{t \nearrow T_{\max}} \left\| u(\cdot,t) \right\|_{L^\infty( \Omega)} <\infty,
\end{equation*}
which contradicts to \eqref{eq-5-2}. Therefore, $T_{\max}=\infty$ and the theorem is proved.

\medskip

(2) By (1), $T_{\max}=\infty$. By Proposition \ref{mass-persistence-prop} (2), there is $\delta>0$ such that
\begin{equation*}
    v(x,t)\ge \delta \quad \forall\,\, x\in\Omega,\,\, t\in (0,\infty).
\end{equation*}
By Theorem \ref{main-thm2}(2), there is $\bar p>n$ such that
\begin{equation*}
    \sup_{0\le t<\infty}\|u(\cdot,t)\|_{L^{\bar p}}<\infty.
\end{equation*}
It follows from the arguments in (1)  that
\begin{equation*}
    \sup_{0<t<\infty} \left\| u(\cdot,t) \right\|_{L^\infty( \Omega)} <\infty.
\end{equation*}

\medskip

{ (3) It follows from the arguments of (1)-(2).}
\end{proof}

\end{document}